\documentclass[11pt,reqno]{amsart}
\usepackage{enumerate}
\usepackage{enumitem}
\usepackage{comment}
\usepackage{mathrsfs}
\usepackage{url}
\usepackage{cite}
\usepackage{amsmath}
\usepackage{graphicx}
\usepackage{amsrefs}
\usepackage{upgreek}
\usepackage{tikz}
\usepackage{amsthm}
\usepackage{amsfonts,amssymb}
\usetikzlibrary{arrows,shapes,trees,backgrounds}
\usepackage{fancyhdr}
\usepackage{amsfonts, amsmath, amssymb, amscd, bm, cancel}

\usepackage[
linktocpage=true,colorlinks,citecolor=magenta,linkcolor=blue,urlcolor=magenta]{hyperref}
\usepackage[margin=1.1in]{geometry}
\newtheorem{proposition}{Proposition}[section]
\newtheorem{theorem}[proposition]{Theorem}
\newtheorem{problem}[proposition]{Problem}
\newtheorem{lemma}[proposition]{Lemma}
\newtheorem{corollary}[proposition]{Corollary}
\newtheorem{definition}[proposition]{Definition}

\newtheorem{assumption}[proposition]{Assumption}
\theoremstyle{definition}

\theoremstyle{remark}
\newtheorem{remark}[proposition]{Remark}
\numberwithin{equation}{section}
\allowdisplaybreaks

\begin{document}
	
	\title[Horospherical $p$-Christoffel-Minkowski problem]{The horospherical $p$-Christoffel-Minkowski problem in hyperbolic space}
	\author[T. Luo]{Tianci Luo}
	\author[Y. Wei]{Yong Wei}
	\address{School of Mathematical Sciences, University of Science and Technology of China, Hefei 230026, P.R. China}
	\email{\href{mailto:Luo_tianci@mail.ustc.edu.cn}{Luo\_tianci@mail.ustc.edu.cn}}
	\email{\href{mailto:yongwei@ustc.edu.cn}{yongwei@ustc.edu.cn}}
	\thanks{The work was supported by National Key Research and Development Program of China 2021YFA1001800 and 2020YFA0713100, and the Fundamental Research Funds for the Central Universities.}
	
	\begin{abstract}
		The horospherical $p$-Christoffel-Minkowski problem was posed by Li and Xu in \cite{Li-Xu} as a problem prescribing the $k$-th horospherical $p$-surface area measure of $h$-convex domains in hyperbolic space $\mathbb{H}^{n+1}$. It is a natural generalization of the classical $L^p$ Christoffel-Minkowski problem in the Euclidean space $\mathbb{R}^{n+1}$. In this paper, we consider a fully nonlinear equation associated with the horospherical $p$-Christoffel-Minkowski problem. We establish the existence of a uniformly $h$-convex solution under appropriate assumptions on the prescribed function. The key to the proof is the full rank theorem, which we will demonstrate using a viscosity approach based on the idea of Bryan-Ivaki-Scheuer \cite{BIS23}.
		
		When $p=0$, the horospherical $p$-Christoffel-Minkowski problem in $\mathbb{H}^{n+1}$ is equivalent to a Nirenberg-type problem on $\mathbb{S}^n$ in conformal geometry.  Therefore, our result implies the existence of solutions to the Nirenberg-type problem.
	\end{abstract}
	
	\maketitle

	
	\section{Introduction}
 The $L^p$ Christoffel-Minkowski problem in $\mathbb{R}^{n+1}$ is a problem of prescribing the $k$-th $p$-surface area measure (introduced by Lutwak \cite{Lut93}) of a convex body in $\mathbb{R}^{n+1}$ for $1\leq k\leq n-1$ and $p\geq 1$. In smooth category, the problem corresponds to finding a convex solution $\varphi$ to the fully nonlinear PDE
 \begin{equation}\label{s1.LpR}
			\sigma_k(D^2\varphi+\varphi g_{\mathbb{S}^n})=\varphi^{p-1}f\quad \text{on} \quad \mathbb{S}^n,
\end{equation}
where $f\in C^\infty(\mathbb{S}^n)$ is a prescribed function and the convexity means
\begin{equation*}
    \left(D^2\varphi+\varphi g_{\mathbb{S}^n}\right)>0 \quad \text{on} \quad \mathbb{S}^n.
\end{equation*}
When $p=1$, the problem corresponds to the classical Christoffel-Minkowski problem and the existence result of a convex solution was obtained by Guan-Ma \cite{GM03} under some assumption on the prescribing function $f$. See also \cite[Part 1]{Guan} for more introduction on the Christoffel-Minkowski problem.  Equation \eqref{s1.LpR} was further studied by Hu-Ma-Shen \cite{HMS04} in the case $p\geq k+1$ and by Guan-Xia \cite{GX18} under the evenness assumption in the case $1<p<k+1$. A key tool in the proof is the full rank theorem proved in \cites{GM03,HMS04} under the assumption is that $f^{-\frac{1}{k+p-1}}$ is spherical convex, i.e.
\begin{equation}\label{s1.crt}
			\left(D^2f^{-\frac{1}{k+p-1}}+f^{-\frac{1}{k+p-1}}g_{\mathbb{S}^n}\right)\geq 0.
\end{equation}
The curvature flow approaches to the Christoffel-Minkowski problem were also provided recently by Bryan-Ivaki-Scheuer \cite{BIS23b} and Ivaki \cite{Iva19}, Zhang \cite{Zha24}.
	
It is a natural question to ask for the extension of the $L^p$ Christoffel-Minkowski problem to the hyperbolic space  $\mathbb{H}^{n+1}$.  Building on the work of Andrews-Chen-Wei \cite{And20} (and an earlier related work  Espinar-G\'alvez-Mira \cite{Esp09}), Li-Xu \cite{Li-Xu} introduced a summation of two sets in hyperbolic space which they called the hyperbolic $p$-sum and established $p$-Brunn-Minkowski  theory in $\mathbb{H}^{n+1}$. Utilizing the hyperbolic $p$-sum and the modified quermassintegrals introduced in \cite{And20}, Li-Xu further introduced the $k$-th horospherical $p$-surface area measures for smooth uniformly $h$-convex bounded domains in $\mathbb{H}^{n+1}$, and proposed the horospherical $p$-Christoffel-Minkowski problem.
	
	We first recall some basic notations briefly, see \cite[\S 5]{And20}, \cite[\S 2-4]{Li-Xu} and Section \ref{sec2} for details. A bounded domain $K$ in the hyperbolic space $\mathbb{H}^{n+1}$ is said to be $h$-convex if it is convex by the horospheres. When $K$ is smooth, $h$-convexity is equivalent to that all the principal curvatures of $\partial K$ satisfy $\kappa_i\geq 1$ for all $i=1,\cdots,n$. We also say that a smooth bounded domain $K\subset \mathbb{H}^{n+1}$ is uniformly $h$-convex if the principal curvatures of $\partial K $ satisfy $\kappa_i>1$ for all $i=1,\cdots,n$. We identify the hyperbolic space with the future time-like hyperboloid in the Minkowski space $\mathbb{R}^{n+1,1}$.  For a smooth uniformly $h$-convex domain $K$ in $\mathbb{H}^{n+1}\subset \mathbb{R}^{n+1,1}$, let $X$ and $\nu$ be the position vector and unit outward normal of $\partial K$. Since $\langle X,X\rangle =-1$ and $\langle X,\nu\rangle=0$ on $\partial K$, there exists $u\in\mathbb{R}$ and $x\in\mathbb{S}^n$ such that
 \begin{equation*}
     X-\nu=e^{-u}\left(x,1\right).
 \end{equation*}
 The horospherical support Gauss map $G:\partial K\to\mathbb{S}^n$ is defined by $G(X)=x$ and here $u\in C^\infty(\mathbb{S}^n)$ is called the horospherical support function of $K$. For a uniformly $h$-convex domain $K$, the map $G$ is a diffeomorphism and $\partial K$ can be reparametrized by $\mathbb{S}^n$ via the inverse map of $G$. In fact, the position vector of $\partial K$ can be expressed as $X:\mathbb{S}^n\to \partial K$ satisfying
 \begin{equation}\label{s1.X}
		X(x)=\dfrac{1}{2}\varphi(-x,1)+\dfrac{1}{2}\left(\dfrac{|D\varphi|^2}{\varphi}+\dfrac{1}{\varphi}\right)(x,1)-(D\varphi,0),
	\end{equation}
where $\varphi=e^u\in C^\infty(\mathbb{S}^n)$ is also called the horospherical support function of $K$. We define a symmetric $2$-tensor $A[\varphi]$ on $\mathbb{S}^n$ as
	\begin{equation}\label{s1.Aij}
		A_{ij}[\varphi]=\varphi_{ij}-\dfrac{|D\varphi|^2}{2\varphi}\sigma_{ij}+\dfrac{1}{2}(\varphi-\dfrac{1}{\varphi})\sigma_{ij},
	\end{equation}
	where $\sigma_{ij}$	denotes the components of the round metric $g_{\mathbb{S}^n}$.  Then the Weingarten tenor $h_i^j$ of $\partial K$ satisfies \cite[(5.7)]{And20}
 \begin{equation}\label{s1.Aij2}
     \left(h_i^j-\delta_i^j\right)\varphi(x)A_{jk}[\varphi(x)]=\sigma_{ik}.
 \end{equation}
 It follows that for a smooth function $u\in C^\infty(\mathbb{S}^n)$, the map \eqref{s1.X} defines an embedding to the boundary of a uniformly $h$-convex domain if and only if $A[\varphi]$ is positive definite.

For any $p>0$, the hyperbolic $p$-sum $K+_pL$ of two uniformly $h$-convex bounded domains $K$ and $L$ in $\mathbb{H}^{n+1}$ is defined through their horospherical support functions (see \cite[Section 4]{Li-Xu}). The $p$-mixed $k$-th modified quermassintegrals for two smooth uniformly $h$-convex bounded domains $K$, $L\subset\mathbb{H}^{n+1}$ are defined by
	\begin{equation}\nonumber
		\tilde{W}_{p,k}(K,L)=\lim_{t\to0^+}\dfrac{\tilde{W}_k(K+_pt\cdot L)-\tilde{W}_k(K)}{t},\quad1\le k\le n.
	\end{equation}
	Here $\tilde{W}_k(K)$ denotes the modified quermassintegral for $K$ which was introduced by Andrews-Chen-Wei \cite{And20} and further studied in \cites{GZ24,Hu20}. It was shown by Li-Xu \cite[Section 5]{Li-Xu} that $\tilde{W}_{p,k}(K,L)$ has the following integral representation:
	\begin{equation}\nonumber
		\tilde{W}_{p,k}(K,L)=\frac{1}{p}\int_{\mathbb{S}^n}\varphi^p_LdS_{p,k}(K,x),
	\end{equation}
	where $\varphi_L=e^{u_L}$ is the horospherical support function of $L$ and $dS_{p,k}(K,x)$ is the $k$-th horospherical $p$-surface area measure of $K$  defined by
	\begin{equation}\nonumber
		dS_{p,k}(K,x)=\frac{1}{C_n^{n-k}}\varphi_K^{-p-k}\sigma_{n-k}\Big(A[\varphi_K]\Big)d\sigma_{\mathbb{S}^n}.
	\end{equation}
 Particularly, we call $dS(K,x):=dS_{0,0}(K,x)$ the horospherical surface area measure of $K$.

	The horospherical $p$-Christoffel-Minkowski problem in $\mathbb{H}^{n+1}$ introduced by Li and Xu \cite{Li-Xu} is a problem prescribing the $k$-th horospherical $p$-surface area measure of $h$-convex bounded domains and can be stated as following.
	\begin{problem}[see Problem 5.2 in \cite{Li-Xu}]\label{problem}
		Let $n\ge2$ and $1\le k\le n-1$ be integers, and let $p$ be a real number. For a given smooth positive function $f(x)$ defined on $\mathbb{S}^n$, what are the necessary and sufficient conditions for $f(x)$, such that there exists a smooth uniformly h-convex bounded domain $K\subset\mathbb{H}^{n+1}$ satisfying
		\begin{equation}\nonumber
			dS_{p,k}(K,x)=f(x)d\sigma_{\mathbb{S}^n}.
		\end{equation}
  That is, finding a smooth positive solution $\varphi$ to
  \begin{equation}\label{s1.prob-eq}
      \frac{1}{C_n^{n-k}}\varphi^{-p-k}\sigma_{n-k}\Big(A[\varphi]\Big)=f(x)
  \end{equation}
such that $A[\varphi]>0$   on $\mathbb{S}^n$.

	\end{problem}
	The horospherical $p$-Christoffel-Minkowski problem has been studied in \cite{Li-Xu} using a curvature flow approach, and Li-Wan \cite{Li23} has studied the problem in the hyperbolic plane (i.e. $n=1$). The case $k=1, p=-n$ of \eqref{s1.prob-eq} was also studied by Chen \cite{Chen}
	
For convenience of the notations, let
 \begin{equation*}
     k'=n-k,\quad  p'=p+n \quad  \text{and}\quad  f'=C_n^{n-k}f.
 \end{equation*}
 The equation \eqref{s1.prob-eq} is equivalent to
 \begin{equation}\label{s1.prob-eq2}
     \sigma_{k'}\left(A[\varphi]\right)=\varphi^{p'-k'}f'.
 \end{equation}
 If there is no confusion, we rewrite \eqref{s1.prob-eq2} equivalently as
	\begin{equation}\label{equa}
			\sigma_{k}(A[\varphi])=~\varphi^{p-k}f\qquad \text{on}\quad \mathbb{S}^n,
	\end{equation}
  where $k=1,\cdots,n-1$. 	We say that $\varphi$ is a $h$-convex solution of  \eqref{equa} if $A[\varphi]\ge0$ on $\mathbb{S}^n$, and $\varphi$ is a uniformly $h$-convex solution of \eqref{equa} if $A[\varphi]>0$ on $\mathbb{S}^n$.
	
	To solve the equation \eqref{equa}, we make the following assumptions on $f$.
	\begin{assumption}\label{assum}
		Let $n\ge2$ and $1\le k\le n-1$ be integers. Let $p\geq 0$ be a real number and $f$ be a smooth positive function on $\mathbb{S}^n$.
		
		(1) If $p=0$, we assume that
		\begin{equation*}
			D^2(f^{ -\frac{1}{k}})-|Df^{-\frac{1}{k}}|g_{\mathbb{S}^n}+\dfrac{f^{-\frac{1}{k}}}{2+8\max\limits_{\mathbb{S}^n}\left(\dfrac{f}{C^k_n}\right)^{\frac{1}{k}}}g_{\mathbb{S}^n}\ge0.
		\end{equation*}
		
		(2) If $0<p\le\frac{k}{2}$, we assume that
		\begin{equation*}
			D^2(f^{-\frac{1}{k}})-\frac{k+2p}{k}|Df^{-\frac{1}{k}}|g_{\mathbb{S}^n}+\dfrac{p^2}{k^2}f^{-\frac{1}{k}}g_{\mathbb{S}^n}\ge0.
		\end{equation*}
		
		(3) If $\frac{k}{2}<p<k$, we assume that
		\begin{equation*}
			D^2(f^{-\frac{1}{k}})-\dfrac{(3k-2p)^2}{p(2p-k)}\dfrac{|Df^{-\frac{1}{k}}|^2}{2f^{-\frac{1}{k}}}g_{\mathbb{S}^n}+\dfrac{1}{2}\dfrac{p}{k}f^{-\frac{1}{k}}g_{\mathbb{S}^n}\ge0.
		\end{equation*}
		
		(4) If $k\le p<2k$, we assume that
		\begin{equation*}
			D^2(f^{-\frac{1}{p}})-\dfrac{|Df^{-\frac{1}{p}}|^2}{2f^{-\frac{1}{p}}}g_{\mathbb{S}^n}+\frac{1}{2}f^{-\frac{1}{p}}g_{\mathbb{S}^n}\ge0.
		\end{equation*}

        (5) If $p=2k$, we assume that
		\begin{equation*}
            \max\limits_{\mathbb{S}^n}f<\dfrac{C_n^k}{2^k}\quad \text{and}\quad
			D^2(f^{-\frac{1}{p}})-\dfrac{|Df^{-\frac{1}{p}}|^2}{2f^{-\frac{1}{p}}}g_{\mathbb{S}^n}+\frac{1}{2}f^{-\frac{1}{p}}g_{\mathbb{S}^n}\ge0.
		\end{equation*}

        (6) If $p>2k$, we assume that $2\le k\le n-1$,
		\begin{equation*}
            \min\limits_{\mathbb{S}^n}f\le\gamma_p:=\dfrac{k^k(p-2k)^{\frac{p-2k}{2}}}{p^{\frac{p}{2}}}C^k_n\quad \text{and}\quad
			D^2(f^{-\frac{1}{p}})-\dfrac{|Df^{-\frac{1}{p}}|^2}{2f^{-\frac{1}{p}}}g_{\mathbb{S}^n}+\frac{k}{p}f^{-\frac{1}{p}}g_{\mathbb{S}^n}\ge0.
		\end{equation*}
	\end{assumption}
	\begin{remark}
		For a given smooth and positive function $h(x)$ defined on $\mathbb{S}^n$, it is straightforward to check that
		$f=(h^{-\frac{1}{k}}+C)^{-k}$ satisfies the conditions in Assumption \ref{assum} for $C$ large enough. Moreover, if $f(x)$ is $C^2$ closed to a constant function, then $f(x)$ also satisfies the Condition (1),(2) and (3). Conditions (4) (5) in Assumption \ref{assum} is a more natural requirement that fits the assumption \ref{s1.crt} in the classical $L^p$ Christoffel-Minkowski problem.
	\end{remark}
	
	We say that a function $f$ on $\mathbb{S}^n$ is even if it satisfies $f(x)=f(-x)$ for any $x\in\mathbb{S}^n$. Now we state our main result.
	\begin{theorem}\label{Main}
		Let $n\ge2$ and $1\le k\le n-1$ be integers. Let $p\geq 0$ be a real number. Given a smooth positive even function $f(x)$ on $\mathbb{S}^n$. If $f(x)$ satisfies the Assumption \ref{assum}, then equation \eqref{equa} has a smooth, even and uniformly h-convex solution $\varphi>1$.
	\end{theorem}
	
 Theorem \ref{Main} establishes a sufficient condition for the corresponding case of Problem \ref{problem} through the relationship between equations \eqref{equa} and \eqref{s1.prob-eq}.
	
 \begin{remark}
		The case $k=1,p=0$ of Theorem \ref{Main} was recently established by Chen \cite{Chen}. For the case $p>0$, Li-Xu \cite[Theorem 7.3]{Li-Xu} also demonstrated a similar existence result under the same assumption using the curvature flow approach (The condition in (2) of Assumption \ref{assum} is slightly better than that in \cite[Assumption 7.1]{Li-Xu} for the corresponding case). However, they could only prove the existence of an even, uniformly $h$-convex solution $\varphi$ satisfying
  \begin{equation}\label{s1.lixu-equ}
      \sigma_{k}(A[\varphi])=\gamma\varphi^{p-k}f
  \end{equation}
  on $\mathbb{S}^n$ for some unknown constant $\gamma>0$. For $p=0$, Li-Xu \cite[Theorem 7.3]{Li-Xu} assumed further that the prescribed function $f$ is a constant. Hence our result provides more precise information.		
	\end{remark}
	
	To find a uniformly $h$-convex solution, a key tool to handle the convexity is the full rank theorem for the equation \eqref{equa}. We will use a viscosity approach based on the idea of Bryan-Ivaki-Scheuer \cite{BIS23} where they provided a simplified proof of the full rank theorem for the classical Christoffel-Minkowski problem in $\mathbb{R}^{n+1}$ using a viscosity approach. The key is a viscosity differential inequality for the minimal eigenvalue of a symmetric $2$-tensor in Brendle-Choi-Daskalopoulos \cite[Lemma 5]{BCD17}, and we use it to derive that the minimal eigenvalue of $A[\varphi]$ satisfies a linear differential inequality in a viscosity sense and the latter allows us to use the strong maximum principle. This avoids the use of a nonlinear test function $\psi=\sigma_{l+1}(A[\varphi])$ as in Guan-Ma's work \cite{GM03}. In the proof of the full rank theorem for \eqref{s1.LpR} in the Euclidean space, a key assumption is that the symmetric $2$-tensor $D^2\varphi+\varphi g_{\mathbb{S}^n}$ is a Codazzi tensor. However, the tensor $A[\varphi]$ for the horospherical $p$-Christoffel-Minkowski problem in  $\mathbb{H}^{n+1}$ is not Codazzi and is more complicated. This complexity necessitates handling additional terms when commuting covariant derivatives, and leveraging basic properties of elementary symmetric functions. 
	
	When $p=0$, the horospherical  $p$-Christoffel-Minkowski problem \eqref{equa} reduces to  the \textit{generalized Christoffel problem} considered in Espinar-G\'alvez-Mira \cite[\S 5]{Esp09}. Define the hyperbolic curvature radii of a uniformly $h$-convex domain $K\subset\mathbb{H}^{n+1}$ by
 \begin{equation*}
     \mathcal{R}_i=\frac{1}{\kappa_i-1},\qquad i=1,\cdots,n,
 \end{equation*}
where $\kappa=(\kappa_1,\cdots,\kappa_n)$ are principal curvatures of $\partial K$.  The \textit{generalized Christoffel problem} introduced in \cite{Esp09} states as
\begin{problem}[See \S 5 in \cite{Esp09}]\label{s1.prob}
Let $n\geq 2$ and $1\leq k\leq n$ be integers. Given a smooth function $f$ on $\mathbb{S}^n$. Find if there exists a uniformly $h$-convex bounded domain $K\subset \mathbb{H}^{n+1}$ such that
\begin{equation}\label{s1.prob2-eq}
    \sigma_k\left(\mathcal{R}_1,\cdots,\mathcal{R}_n\right)=f.
\end{equation}
\end{problem}
We note that the notation in \cite{Esp09} is slightly different from ours. In that paper the horospherical convex domains are those which are intersection of complements of horoballs everywhere, while we deal with domains which are intersections of horoballs.  But the idea of introducing the \textit{generalized Christoffel problem}  is essentially the same as the statement in Problem \ref{s1.prob}.

By the equation \eqref{s1.Aij2}, we know that $\mathcal{R}_i$ are just eigenvalues of the symmetric $2$-tensor $\varphi A[\varphi]$ on $\mathbb{S}^n$. The equation \eqref{s1.prob2-eq} is equivalent to
\begin{equation}
    \sigma_k\left(A[\varphi]\right)=\varphi^{-k}f,\qquad \text{on} \quad \mathbb{S}^n,
\end{equation}
which corresponds to \eqref{equa} for $p=0$. Moreover, it is shown in \cite[Corollary 27]{Esp09} that Problem \ref{s1.prob} is tightly linked to the question of prescribing a given functional of the eigenvalues of the Schouten tensor for a conformal metric on $\mathbb{S}^n$.

Recall that the  Schouten tensor of a Riemannian metric $g$ on a manifold $M^n$ is a symmetric $2$-tensor defined by
	\begin{equation}
		Sch_g:=\frac{1}{n-2}\left(Ric_g-\frac{S(g)}{2(n-1 )}g\right),		
	\end{equation}
	where $Ric_g$ and $S(g)$ stand for the Ricci and scalar curvatures of $g$. The Schouten tensor completely determines the curvature tensor for a conformally flat metric and plays a key role in the conformal geometry.  The problem of prescribing a certain functional of the eigenvalues of the Schouten tensor for a conformal metric on a Riemannian manifold has been studied extensively. See \cites{Ch1,Ch2,Gu07,Li03,GW,STW,GE,Ch3,JL,Li95,Vi} for instance and specially the book \cite[Part 2]{Guan}. The problem of finding a conformal metric on $\mathbb{S}^n$ with a prescribed $\sigma_k$-curvature $S_k=\sigma_k(Sch_g)$ has been specifically treated, as a generalization of the Nirenberg problem.

Let $M=\partial K$ be a smooth uniformly $h$-convex hypersurface in $\mathbb{H}^{n+1}$ with horospherical support function $\varphi=e^u\in C^\infty(\mathbb{S}^n)$. The \textit{horospherical metric} is a conformal metric $g=\varphi^{-2}g_{\mathbb{S}^n}$ on the sphere $\mathbb{S}^n$. It is calculated that the hyperbolic curvature radii $\mathcal{R}_i$ of $M$ and the eigenvalues $\lambda_i$ of $Sch_g$ on $\mathbb{S}^n$ satisfy (see \cite[\S 7]{And20} and \cite[\S 5]{Esp09})
\begin{equation*}
    \mathcal{R}_i=\lambda_i-\frac{1}{2},\quad i=1,\cdots,n,
\end{equation*}
and so
\begin{equation*}
\sigma_k(\mathcal{R}_1,...,\mathcal{R}_n)=\sum_{j=0}^{k}C_{n-j}^{k-j}(-1)^{k-j}2^{j-k}\sigma_j(\lambda_1,...,\lambda_n).
\end{equation*}
By this relation, the \textit{generalized Christoffel problem} is equivalent to the Nirenberg-type problem on $\mathbb{S}^n$ in which the following linear combination of the $\sigma_k$-curvatures is prescribed
\begin{equation}\label{s1.conf}
	\sum_{j=0}^{k}c_j\sigma_j(Sch_g)=f,\qquad c_j:=C_{n-j}^{k-j}(-1)^{k-j}2^{j-k}.
\end{equation}
Since the Schouten tensor only makes sense for dimension $n\geq 3$, our main result in Theorem \ref{Main} implies the following corollary.
\begin{corollary}
	Let $f$ be a smooth and positive function on $\mathbb{S}^n$ with $n\ge3$, if f satisfies Condition (1) in Assumption \ref{assum}, then there exists a conformal metric $g=\varphi^{-2}g_{\mathbb{S}^n}$ on $\mathbb{S}^n$ such that its Schouten tensor $Sch_g$ satisfies the Nirenberg-type equation \eqref{s1.conf}. Additionally, $g$ satisfies the regularity condition that $2Sch_g-g$ is positive definite.
\end{corollary}
	
	The paper is organized as follows. In Section \ref{sec2}, we present some preliminaries for our proof, including the basic properties of symmetric functions, fundamental concepts and properties of $h$-convex domains and $h$-convex hypersurfaces in $\mathbb{H}^{n+1}$, and introduce the Nirenberg-type problem in $\mathbb{S}^n$. In Section \ref{sec3}, we establish a priori estimates for solutions to the equation \eqref{equa}. Section \ref{sec4} contains the proof of the full rank theorem. We employ degree theory to prove Theorem \ref{Main} in Section \ref{sec5}.

	\section{Preliminaries}\label{sec2}
	\subsection{The elementary symmetric functions}
	
	We first review some basic properties of elementary symmetric functions. See \cite[\S 2]{Guan14} for more details.
	
	For any integer $k=1,2,\cdots,n$ and $\lambda=(\lambda_{1},\cdots,\lambda_{n})\in\mathbb{R}^{n}$, the $k$-th elementary symmetric function is defined by
	\begin{eqnarray}\label{sigma}
		\sigma_{k}(\lambda)=\sum\limits_{1\le i_1<i_2<\cdots < 	i_k\le n}
		\lambda_{{i}_{1}}\lambda_{{i}_{2}}\cdots\lambda_{{i}_{k}}.
	\end{eqnarray}
	We also set $\sigma_0(\lambda)=1$ and $\sigma_m(\lambda)=0$ for $m>n$. Denote $\sigma_k(\lambda|i)$ for the symmetric function with $\lambda_i=0$ and $\sigma_k(\lambda|ij)$ the symmetric function with $\lambda_i=\lambda_j=0$.
 \begin{lemma}\label{s2.lem1}
     Let $k=1,2,\cdots,n$ and $\lambda=(\lambda_{1},\cdots,\lambda_{n})\in\mathbb{R}^{n}$. Then
     \begin{align*}
         &\sigma_k(\lambda)=\sigma_k(\lambda|i)+\lambda_i\sigma_{k-1}(\lambda|i),\quad \forall~i=1,\cdots,n\\
         &\sum_i\lambda_i\sigma_{k-1}(\lambda|i)=k\sigma_k(\lambda)\\
         &\sum_i\sigma_k(\lambda|i)=(n-k)\sigma_k(\lambda).
     \end{align*}
 \end{lemma}

 The definition of $\sigma_k$ can be extended to symmetric matrices. Let $A$ be an $n \times n$ symmetric matrix $A$ with eigenvalues $\lambda(A)=(\lambda_1(A)$, ..., $\lambda_n(A))$. We define $\sigma_{k}(A)=\sigma_{k}(\lambda(A))$. Similarly, denote by $\sigma_k(A|i)$ the symmetric function with $A$ deleting the $i$-row and $i$-column, and $\sigma_k(A|ij)$ the symmetric function with $A$ deleting the $i,j$-rows and $i,j$-columns. We have  (see e.g. \cite[\S 2]{CGLS22})
\begin{lemma}\label{AA}
Let $A=(A_{ij})$ be a diagonal $n \times n$ matrix. Then
\begin{equation*}
    \sigma_k^{ij}(A)=\frac{\partial\sigma_{k}}{\partial A_{ij}}(A)=\begin{cases}
    \sigma_{k-1}(A|i),& \quad if\quad i=j\\[2ex]
    0,\qquad &\quad if\quad i\neq j
    \end{cases}
\end{equation*}
and
\begin{equation*}
    \sigma_k^{ij,rs}(A)=\frac{\partial^2\sigma_{k}}{\partial A_{ij}\partial A_{rs}}(A)=
			\begin{cases}
				\displaystyle \sigma_{k-2}(A|ir), & \text{if}\quad i=j, r=s, i\neq r;\\[2ex]
				\displaystyle -\sigma_{k-2}(A|ij), & \text{if} \quad i\neq j, r=j, s=i;\\[2ex]
				\displaystyle 0, & \text{otherwise}.
			\end{cases}
\end{equation*}
\end{lemma}

 Let $\Gamma^+:=\{\lambda\in\mathbb{R}^n,~\lambda_i>0,~\forall~i=1,\cdots,n\}$ be the positive cone in $\mathbb{R}^n$ and $\Gamma_k$ be the connected component of $\{\lambda\in\mathbb{R}^n:~\sigma_k(\lambda)>0\}$ containing the positive cone. It is know that $F=\sigma_k^{1/k}$ is concave in $\Gamma_k$ and inverse-concave in $\Gamma_+$. Here the concavity means that  (\cite[Theorem 3.2]{Ge15})
 \begin{equation*}
     F^{ij,rs}\xi_{ij}\xi_{rs}\leq 0,\quad \forall~ (\xi_{ij})\in\mathbb{R}^{n\times n}.
 \end{equation*}
 The inverse-concavity of $F=\sigma_k^{1/k}$ implies that at an invertible symmetric $A=(A_{ij})$,
	\begin{equation*}
		F^{ij,rs}\xi_{ij}\xi_{rs}+2F^{ir}A^{js}\xi_{ij}\xi_{rs}\ge2F^{-1}(F^{ij}\xi_{ij})^2,\qquad \forall~ (\xi_{ij})\in\mathbb{R}^{n\times n}.
	\end{equation*}
	where $(A^{ij})$ is the inverse of $(A_{ij})$,  see \cite[(3.49))]{A91}. This in turn implies the following lemma.
	\begin{lemma}\label{sig}
		At any invertible symmetric matrix $A=(A_{ij})$, the $k$-th elementary symmetric function $\sigma_k(A)$ satisfies
		\begin{equation}
			\sigma_k^{ij,rs}\xi_{ij}\xi_{rs}+2\sigma_k^{ir}A^{js}\xi_{ij}\xi_{rs}\ge\dfrac{k+1}{k}\frac{(\sigma_k^{ij}\xi_{ij})^2}{\sigma_k},\qquad\forall~ (\xi_{ij})\in\mathbb{R}^{n\times n},
		\end{equation}
		where $(A^{ij})$ is the inverse of $(A_{ij})$.
	\end{lemma}
	
	\subsection{Horospherically convex domains and hypersurfaces}
	In this subsection, we will collect basic concepts and properties of horospherically convex hypersurfaces in $\mathbb{H}^{n+1}$. We refer to \cite{And20},\cite{Li-Xu} for details.
	
	We shall work in the hyperboloid model of $\mathbb{H}^{n+1}$, which we identify $\mathbb{H}^{n+1}$ with the future time-like hyperboloid in the Minkowski space $\mathbb{R}^{n+1,1}$. The horospheres are hypersurfaces in $\mathbb{H}^{n+1}$ with constant principal curvatures equal to $1$ everywhere and can be parametrized by $\mathbb{S}^n\times \mathbb{R}$:
	\begin{eqnarray*}
		H_{x}(r)=\{X \in \mathbb{H}^{n+1}~|~\langle X, (x, 1)\rangle=-e^r\},
	\end{eqnarray*}
	where $x \in \mathbb{S}^n$ and $r \in \mathbb{R}$ represents the signed geodesic distance from the ``north pole" $N =(0, 1)\in \mathbb{H}^{n+1}$. 	The interior region (called a horoball) is denoted by
	\begin{eqnarray}\label{Bx}
		B_{x}(r)=\{X \in \mathbb{H}^{n+1}~|~ 0>\langle X, (x,
		1)\rangle>-e^r\}.
	\end{eqnarray}
	
	A region $K\subset \mathbb{H}^{n+1}$ (or its boundary $\partial K$) is horospherically convex (or $h$-convex for short) if every boundary point $p$ of
	$\partial K$ has a supporting horo-ball, i.e. a horo-ball $B$ such that $K\subset {B}$ and $p \in \partial B$. When $K$ is smooth, it is $h$-convex if
	and only if the principal curvatures of $\partial K$ are greater than or equal to $1$. For a smooth compact $K$, we say that $K$ (or $\partial K$) is uniformly $h$-convex if its principal curvatures are greater than $1$.

	Let $K$ be a bounded $h$-convex domain in $\mathbb{H}^{n+1}$. Then for each $x\in \mathbb{S}^n$, we define the \textit{horospherical support function} of $K$ (or $\partial K$) in direction $x$ by
		\begin{equation*}
			u_{K}(x):= \inf\{s\in\mathbb{R}:K\subset {B}_x(s)\},\quad x\in \mathbb{S}^n.
		\end{equation*}
		Alternatively, we have
		\begin{equation*}
			u_{K}(x)= \sup\{\log(-\langle X,(x,1)\rangle ):X\in K\},\quad x\in \mathbb{S}^n.
		\end{equation*}
		The horospherical support function completely determines a horospherically convex domain $\Omega$ as an intersection of horoballs:
      \begin{equation*}
        K=\bigcap_{x\in\mathbb{S}^n}B_x(u_K(x)).
      \end{equation*}
       We will write $u(x)$ for $u_{K}(x)$ if there is no confusion about the choice of the domain.

	Let $\nu$ be the outward unit normal of $\partial K$. Since $\langle X(p),X(p)\rangle =-1$, $\langle \nu(p),\nu(p)\rangle =1$ and $\langle X(p),\nu(p)\rangle =0$ for each $p\in\partial K$, we have that $X-\nu$ is a null vector and so there exists $u\in\mathbb{R}$ and $x\in \mathbb{S}^n$ such that
 \begin{equation*}
     X-\nu=e^{-u}\left(x,1\right).
 \end{equation*}
 Then the \textit{horospherical Gauss map} $G:\partial K\to \mathbb{S}^n$ is defined by $G(p)=x$ and the function $u=u_K$ is just the horospherical support function. The derivative of the horospherical Gauss map is non-singular if $K$ is uniformly $h$-convex and in this case $G$ is a diffeomorphism from $\partial K$ to $\mathbb{S}^n$.

 Let $M=\partial K$ be the boundary of a smooth uniformly $h$-convex domain. We can reparametrize $M$ by $\mathbb{S}^n$ via the inverse of the horospherical Gauss map $X=G^{-1}:\mathbb{S}^n\to M$. Let $\varphi=e^{u_K}$ which we also call the horospherical support function and is bigger than $1$. We have  (see \cite[(5.5)]{And20})
	\begin{equation}\label{X}
		X(x)=\dfrac{1}{2}\varphi(-x,1)+\dfrac{1}{2}\left(\dfrac{|D\varphi|^2}{\varphi}+\dfrac{1}{\varphi}\right)(x,1)-(D\varphi,0).
	\end{equation}
By the definition, we have
	\begin{equation}\label{x-v}
		X(x)-\nu(x)=e^{-u(x)}(x,1)=\frac{1}{\varphi(x)}(x,1).
	\end{equation}
So
	\begin{equation}\label{s2.nu}
		\nu(x)=\dfrac{1}{2}\varphi(-x,1)+\dfrac{1}{2}\left(\dfrac{|D\varphi|^2}{\varphi}-\dfrac{1}{\varphi}\right)(x,1)-(D\varphi,0).
	\end{equation}
As in \cite[(5.9)]{And20}, we define a symmetric $2$-tensor on $\mathbb{S}^n$:
	\begin{equation*}
		A_{ij}[\varphi]=\varphi_{ij}-\dfrac{|D\varphi|^2}{2\varphi}\sigma_{ij}+\dfrac{1}{2}(\varphi-\dfrac{1}{\varphi})\sigma_{ij}.
	\end{equation*}
	Then
	\begin{equation}\label{Xj}
		\partial_iX(x)=-\sum_{j=1}^{n}A_{ij}[\varphi](e_j,0)+\dfrac{\sum_{j=1}^{n}A_{ij}[\varphi]D_j\varphi}{\varphi}(x,1).
	\end{equation}
	and
	\begin{equation}\label{Rij}
		(h_i^j-\delta_i^j)\varphi(x) A_{jk}[\varphi(x)]=\sigma_{ik},
	\end{equation}
	where $(h_i^j)$ represents the Weingarten matrix of $M$ evaluated at $X(x)$.
	The formula \eqref{Rij} demonstrates that if $\varphi$ is a smooth function on $\mathbb{S}^n$, the map \eqref{X} defines an embedding to the boundary of a uniformly $h$-convex domain if and only if the tensor $A[\varphi]$ is positive definite. Using \eqref{Xj}, we also establish that $g_{ij}=\sum_kA_{ik}[\varphi]A_{kj}[\varphi]$. Since $A[\varphi]$ is positive definite, the area element of $M$ can be expressed as
	\begin{equation}\label{du}
		d\mu=\sqrt{\det g_{ij}}d\sigma=\det A[\varphi]d\sigma.
	\end{equation}
	
	\begin{definition}
		Let $M^n$ be a smooth uniformly h-convex hypersurface in $\mathbb{H}^{n+1}$. We define the hyperbolic curvature radii $\{\mathcal{R}_1,...,\mathcal{R}_n\}$ of $M^n$ at $X(x)$ as
		\begin{eqnarray*}
			\mathcal{R}_i=\dfrac{1}{k_i-1},
		\end{eqnarray*}
		where $\kappa=(k_1,...,k_n)$ are the principal curvatures of $M^n$.
		
		From \eqref{Rij}, we understand that $\{\mathcal{R}_1,...,\mathcal{R}_n\}$ of $M^n$ at $X(x)$ exactly correspond to the eigenvalues of $(\varphi A[\varphi])$ at $x$ respect to the round metric $g_{\mathbb{S}^n}$ on $\mathbb{S}^n$.
	\end{definition}

 \subsection{Relation to the conformal metric on $\mathbb{S}^n$}
	In this section, we mention an interesting connection between $h$-convex hypersurfaces in the hyperbolic space and the conformally flat metric on the sphere $\mathbb{S}^n$. This was essentially observed by Espinar-G\'alvez-Mira \cite{Esp09}.  In fact, they proposed  the following \textit{generalized Christoffel problem}  in $\mathbb{H}^{n+1}$.
	\begin{problem}[See \S 5 in \cite{Esp09}]\label{s2.prob}
	    Let $f:\mathbb{S}^n\to\mathbb{R}$ be a smooth positive function. Find if there exists a smooth uniformly h-convex hypersurface $M^n$ in $\mathbb{H}^{n+1}$ such that
		\begin{equation}\nonumber
			\sigma_k(\mathcal{R}_1,...,\mathcal{R}_n)=\varphi^k\sigma_k(A[\varphi])=f,
		\end{equation}
		where $\{\mathcal{R}_1,...,\mathcal{R}_n\}$ are the hyperbolic curvature radii of $M^n$.
	\end{problem}
	
	This problem is equivalent to the equation \eqref{equa} for $p=0$.  It is showed in \cite{Esp09} that this problem is precisely equivalent to the question of prescribing a given functional of the eigenvalues of the Schouten tensor for a conformal metric on $\mathbb{S}^n$.
	
	\begin{lemma}\cite[\S 7]{And20}
		Let $M^n$ be a smooth uniformly $h$-convex hypersurface in $\mathbb{H}^{n+1}$ with horospherical support function $\varphi=e^u:\mathbb{S}^n\to \mathbb{R}$ and hyperbolic curvature radii $\{\mathcal{R}_1,...,\mathcal{R}_n\}$. Denote $g=\varphi^{-2}g_{\mathbb{S}^n}$ which is a conformally flat metric on $\mathbb{S}^n$ and let $\{\lambda_1,...,\lambda_n\}$ be the eigenvalues of the Schouten tensor $Sch_g$ of the metric $g$ on $\mathbb{S}^n$.  Then
		\begin{equation}\label{sch}
			\mathcal{R}_i=\lambda_i-\frac{1}{2},\quad i=1,\cdots,n.
		\end{equation}
		Moreover, the eigendirections of $Sch_g$ coincide with the principal directions of $M^n$.
	\end{lemma}
	
	By \eqref{sch} we have
	\begin{equation}\nonumber
		\sigma_k(\mathcal{R}_1,...,\mathcal{R}_n)=\sum_{j=0}^{k}C_{n-j}^{k-j}(-1)^{k-j}2^{j-k}\sigma_j(\lambda_1,...,\lambda_n).
	\end{equation}
	This implies that Problem \ref{s2.prob} in $\mathbb{H}^{n+1}$ is equivalent to the Nirenberg-type problem in $\mathbb{S}^n$ for which the following linear combination of the $\sigma_k$-curvatures is prescribed:
	\begin{equation}\nonumber
		\sum_{j=0}^{k}c_j\sigma_j(Sch_g)=f,\qquad c_j:=C_{n-j}^{k-j}(-1)^{k-j}2^{j-k}
	\end{equation}
	In particular, if $k=1$, then
	\begin{align*}
		\sigma_1(\mathcal{R}_1,...,\mathcal{R}_n)=&\sigma_1(\lambda_1,...,\lambda_n)-\dfrac{n}{2}	=\frac{S(g)}{2(n-1)}-\frac{n}{2},
	\end{align*}
	where $S(g)$ is the scalar curvature of $g$. Hence Problem \ref{s2.prob} for $k=1$ is equivalent to the prescribed scalar curvature problem in the conformal class on $\mathbb{S}^n$ which is called Nirenberg-Kazdan-Warner problem.
	
\section{A priori estimates}\label{sec3}
    In this section, we establish a priori estimates for solutions of equation \eqref{equa}. Equation \eqref{equa} becomes uniformly elliptic once $C^{2}$ estimates are established for $\varphi$ and is concave with respect to the second derivatives of $\varphi$. Using the Evans-Krylov theorem and Schauder theory, one can obtain higher derivative estimates for $\varphi$. Therefore, we only need to prove $C^{2}$ estimates for $\varphi$.
	
	We first establish the $C^0$ estimate.
	\begin{lemma}\label{lc0}
		Let $f$ be a smooth, even and positive function satisfying Assumption \ref{assum} and $\varphi>1$ be a smooth, even, h-convex solution of equation \eqref{equa}. There exist constants $C_0,C_1,C_2>1$ depending only on the lower and upper bounds of $f$, $p$, $k$ and $n$ such that
  \begin{enumerate}
      \item If $p=0$, then
		\begin{equation}
			\label{c0}
			1<C_0 \le \min\limits_{\mathbb{S}^n}\varphi \le \max\limits_{\mathbb{S}^n}\varphi \le
			\left(1+2\max\limits_{\mathbb{S}^n}\left(\dfrac{f}{C^k_n}\right)^{\frac{1}{k}}\right)^{\frac{1}{2}}+\left(2\max\limits_{\mathbb{S}^n}\left(\dfrac{f}{C^k_n}\right)^{\frac{1}{k}}\right)^{\frac{1}{2}}.
		\end{equation}
		\item If $p>0$, then
		\begin{equation}\label{s3.c02}
			1<C_1 \le \min\limits_{\mathbb{S}^n}\varphi \le \max\limits_{\mathbb{S}^n}\varphi \le C_2.
		\end{equation}
  \end{enumerate}		
	\end{lemma}
	\begin{proof}
	We first recall a key inequality for the horospherical support function $\varphi=e^u: \mathbb{S}^n\to \mathbb{R}$ of an origin symmetric, uniformly $h$-convex hypersurface in $\mathbb{H}^{n+1}$:
 \begin{equation}\label{pin0}
			\dfrac{1}{2}(\max\limits_{\mathbb{S}^n}\varphi+\dfrac{1}{\max\limits_{\mathbb{S}^n}\varphi}) \le \min\limits_{\mathbb{S}^n}\varphi.
		\end{equation}
  This inequality can be found in the proof of \cite[Lemma 7.2]{Li-Xu}. Suppose that $\varphi>1$ is a smooth, even, $h$-convex solution of equation \eqref{equa}, then $\varphi+\varepsilon$ satisfies $A[\varphi+\varepsilon]>0$ and defines a smooth origin symmetric,  uniformly $h$-convex hypersurface in $\mathbb{H}^{n+1}$. By letting $\varepsilon\to 0$, we see that $\varphi$ still satisfies the estimate \eqref{pin0}.

  Let $\varphi(x_{0})$=$\max\limits_{\mathbb{S}^n}\varphi$ and $\varphi(x_{1})$=$\min\limits_{\mathbb{S}^n}\varphi$, then we have
		\begin{equation}\nonumber
			|D\varphi|(x_0)=|D\varphi|(x_1)=0,\quad	D^2\varphi(x_0)\le0\le D^2\varphi(x_1).
		\end{equation}
		\begin{enumerate}
		    \item If $p=0$, applying the maximum principle from the equation \eqref{equa}, we get
		\begin{equation}
			\nonumber
			\max\limits_{\mathbb{S}^n}\varphi\ge \left(1+2\min\limits_{\mathbb{S}^n}\left(\dfrac{f}{C^k_n}\right)^{\frac{1}{k}}\right)^{\frac{1}{2}} > 1.
		\end{equation}
		Similarly, we have
		\begin{equation}
			\nonumber
			\min\limits_{\mathbb{S}^n}\varphi\le \left(1+2\max\limits_{\mathbb{S}^n}\left(\dfrac{f}{C^k_n}\right)^{\frac{1}{k}}\right)^{\frac{1}{2}}.
		\end{equation}
		Combining with \eqref{pin0} and noting that the function $x+\dfrac{1}{x}$ is increasing when $x\ge1$, we find
		\begin{equation}
			\nonumber
			1<C_0 \le \min\limits_{\mathbb{S}^n}\varphi \le \max\limits_{\mathbb{S}^n}\varphi \le
			\left(1+2\max\limits_{\mathbb{S}^n}\left(\dfrac{f}{C^k_n}\right)^{\frac{1}{k}}\right)^{\frac{1}{2}}+\left(2\max\limits_{\mathbb{S}^n}\left(\dfrac{f}{C^k_n}\right)^{\frac{1}{k}}\right)^{\frac{1}{2}},
		\end{equation}
		where $C_0$ can be taken as
  \begin{equation*}
      C_0=\frac{1}{2}\left(1+2\min\limits_{\mathbb{S}^n}\left(\dfrac{f}{C^k_n}\right)^{\frac{1}{k}}\right)^{\frac{1}{2}}+\frac{1}{2}\left(1+2\min\limits_{\mathbb{S}^n}\left(\dfrac{f}{C^k_n}\right)^{\frac{1}{k}}\right)^{-\frac{1}{2}}>1.
  \end{equation*}
		\item If $0<p\le2k$, applying the maximum principle, we have
		\begin{equation}\nonumber
			\max\limits_{\mathbb{S}^n}\varphi^2\ge 1+2\min\limits_{\mathbb{S}^n}\left(\dfrac{f}{C^k_n}\right)^{\frac{1}{k}}\max\limits_{\mathbb{S}^n}\varphi^{\frac{p}{k}}\ge1+2\min\limits_{\mathbb{S}^n}\left(\dfrac{f}{C^k_n}\right)^{\frac{1}{k}}> 1.
		\end{equation}
		Moreover, we have
		\begin{equation}\nonumber
			\min\limits_{\mathbb{S}^n}\varphi^2\le 1+2\max\limits_{\mathbb{S}^n}\left(\dfrac{f}{C^k_n}\right)^{\frac{1}{k}}\min\limits_{\mathbb{S}^n}\varphi^{\frac{p}{k}}.
		\end{equation}
		If $0<{p}/{k}<2$, we conclude that $\min_{\mathbb{S}^n}\varphi$ is upper bounded. If ${p}/{k}=2$, by our assumption $2\max\limits_{\mathbb{S}^n}\left(\dfrac{f}{C^k_n}\right)^{\frac{1}{k}}<1$,  we also conclude that $\min_{\mathbb{S}^n}\varphi$ is upper bounded. Combining with \eqref{pin0}, we get $1<C_1 \le \min\limits_{\mathbb{S}^n}\varphi \le \max\limits_{\mathbb{S}^n}\varphi \le C_2$, where $C_1,C_2$ depend only on the lower and upper bounds of $f$, $p$, $k$ and $n$.

        \item If $p>2k$, applying the maximum principle, we also have
        \begin{equation}\label{Ineq}
		2\min\limits_{\mathbb{S}^n}\left(\dfrac{f}{C^k_n}\right)^{\frac{1}{k}}\max\limits_{\mathbb{S}^n}\varphi^{\frac{p}{k}}-\max\limits_{\mathbb{S}^n}\varphi^2+1\le 0.
		\end{equation}
        Let $a=2\min\limits_{\mathbb{S}^n}\left(\dfrac{f}{C^k_n}\right)^{\frac{1}{k}}$. By the assumption, we have
        \begin{equation*}
           a\le\dfrac{2k}{p}\Big(\dfrac{p-2k}{p}\Big)^{\frac{p-2k}{2k}}=:a_0.
        \end{equation*}
        We consider a function $g(t)=a t^{\frac{p}{k}}-t^2+1$ on $(1,
        \infty)$, then
        \begin{equation*}
            g'(t)=t\Big(a\dfrac{p}{k}t^{\frac{p}{k}-2}-2\Big).
        \end{equation*}
       It follows that $g(t)$ attains its minimum at
       \begin{equation}\label{s3.xi0}
           t_0=\Big(\dfrac{2k}{a p}\Big)^{\frac{k}{p-2k}}=\left(\frac{a_0}{a}\right)^{\frac{k}{p-2k}}\left(\frac{p}{p-2k}\right)^{\frac{1}{2}}>1
       \end{equation}
       and the minimum value
        \begin{equation*}
            g(t_0)=\Big(\dfrac{2k}{a p}\Big)^{\frac{2k}{p-2k}}\Big(\dfrac{2k}{p}-1\Big)+1\le -\left(\frac{a_0}{a}\right)^{\frac{k}{p-2k}}+1=0.
        \end{equation*}
        Note that $g(t_0)=0$ if $a=a_0$, i.e. the minimum value of $g$ is zero.  Furthermore, it is obvious that $g(1)=a>0$, $g(t)$ is decreasing in $(1,t_0)$ and is increasing on $(t_0,+\infty)$. Moreover, $g(t)$ tends to $+\infty$ as $t\to +\infty$.

        Hence, if $a=a_0$, then $\max\limits_{\mathbb{S}^n}\varphi=t_0$, which combined with \eqref{pin0} gives the estimate \eqref{s3.c02}. If $a<a_0$, then there exist two distinct points $1<t_1<t_2<\infty$ depending only on $a,p$ and $k$ such that
        \begin{equation}\label{s3.xi}
            \{g(t)\le 0\}=\{t\in (t_1,t_2)\}.
        \end{equation}
        So that we have $1<t_1\le\max\limits_{\mathbb{S}^n}\varphi\le t_2$ by \eqref{Ineq}. Combining with \eqref{pin0}, we complete the proof.
		\end{enumerate}
	\end{proof}
    \begin{remark}
        From the proof of Lemma \ref{lc0}, when $p\ge 2k$, it is necessary to add the upper bound condition to the function $f$.
    \end{remark}
	
	It is well-known that $\max_{\mathbb{S}^n}|Du| \le \max_{\mathbb{S}^n}u$ if $u:\mathbb{S}^n\rightarrow \mathbb{R}_+$ is a support function for a smooth, uniformly convex hypersurface in $\mathbb{R}^{n+1}$. A similar conclusion holds for origin symmetric uniformly $h$-convex hypersurfaces in $\mathbb{H}^{n+1}$. the following useful lemma was established by Li-Xu \cite[Lemma 7.3]{Li-Xu}.
	
	\begin{lemma}\label{kc1}
		Let $M$ be a smooth, origin symmetric and uniformly h-convex hypersurface in $\mathbb{H}^{n+1}$ with horospherical support function $\varphi=e^u$, we have
		\begin{equation}
			\label{c1}
			\dfrac{|D\varphi|^2}{\varphi^2}\le1-\dfrac{1}{\varphi^2}\le1,\ \ \forall\ x \in \mathbb{S}^n.
		\end{equation}
	\end{lemma}


By approximation, the estimate \eqref{c1} also holds for even, $h$-convex solution $\varphi$ of the equation \eqref{equa}. Combining Lemma \ref{lc0} and \eqref{c1}, we immediately get the following $C^1$ estimate.
	\begin{lemma}\label{lc1}
        Let $f$ be a smooth, even and positive function satisfying Assumption \ref{assum} and $\varphi>1$ be a smooth even  $h$-convex solution of equation \eqref{equa}. We have
		\begin{equation}
			\nonumber
			|D\varphi|\le C,\ \ \forall x \in \mathbb{S}^n,
		\end{equation}
		where C is a positive constant depending only on the constants in Lemma \ref{lc0}.
	\end{lemma}
	
	Since $A[\varphi]$ is not Codazzi, we need the following lemma to commute the covariant derivatives.
	\begin{lemma}\label{d}
		For each fixed point $x\in\mathbb{S}^n$, we choose a local orthonormal frame so that $(\varphi_{ij})$ is diagonal at $x$, then we have at $x$
		\begin{equation}\label{d1}
			A_{ij\alpha}=A_{ji\alpha},
		\end{equation}
		\begin{equation}\label{d2}
			A_{\alpha \beta i}-A_{i \beta \alpha}=\dfrac{\varphi_{\alpha}}{\varphi}A_{\alpha \alpha}\delta_{i\beta}-\dfrac{\varphi_i}{\varphi}A_{ii}\delta_{\alpha \beta},
		\end{equation}
		and
		\begin{equation}\label{d3}
			\begin{aligned}
				&	A_{ii\alpha\alpha}-A_{\alpha\alpha ii}\\
				=&-\dfrac{\sum_mA_{\alpha\alpha m}\varphi_m}{\varphi}-\left[-\dfrac{2\varphi_{\alpha}^2}{\varphi^2}+\dfrac{|D\varphi|^2}{2\varphi^2}+\dfrac{1}{\varphi}A_{\alpha\alpha}+\dfrac{1}{2}\left(1+\dfrac{1}{\varphi^2}\right)\right]A_{\alpha\alpha}\\
				&+\dfrac{\sum_mA_{iim}\varphi_m}{\varphi}+\left[-\dfrac{2\varphi_{i}^2}{\varphi^2}+\dfrac{|D\varphi|^2}{2\varphi^2}+\dfrac{1}{\varphi}A_{ii}+\dfrac{1}{2}\left(1+\dfrac{1}{\varphi^2}\right)\right]A_{ii}.
			\end{aligned}
		\end{equation}
	\end{lemma}
	\begin{proof}
		Equation \eqref{d1} follows directly from the symmetry of the tensor $A[\varphi]$. To show equation \eqref{d2}, we take the derivatives to \eqref{s1.Aij} and obtain
  \begin{align*}
      A_{\alpha\beta i}=&
			\varphi_{\alpha \beta i}-\Big(\frac{\sum_{l}\varphi_l \varphi_{li}}{\varphi}-\frac{1}{2}\frac{|D\varphi|^2\varphi_i}{\varphi^2}\Big)\delta_{\alpha \beta}
			+\frac{1}{2}\Big(\varphi_i+\frac{\varphi_i}{\varphi^2}\Big)\delta_{\alpha \beta},\\
   A_{i\beta \alpha}=&
			\varphi_{i \beta \alpha}-\Big(\frac{\sum_{l}\varphi_l \varphi_{l\alpha}}{\varphi}-\frac{1}{2}\frac{|D\varphi|^2\varphi_\alpha}{\varphi^2}\Big)\delta_{i \beta}
			+\frac{1}{2}\Big(\varphi_\alpha+\frac{\varphi_\alpha}{\varphi^2}\Big)\delta_{i \beta}.
  \end{align*}
Due to the Ricci identity, we have the commutation formula on $\mathbb{S}^n$
		\begin{align}\label{s3.dalpha}
			\varphi_{\alpha\beta i}-\varphi_{i\beta\alpha}=&\varphi_l\bar{R}_{l\beta\alpha i}
   = \varphi_l\left(\delta_{l\alpha}\delta_{\beta i}-\delta_{li}\delta_{\beta\alpha}\right)\nonumber\\
   =&\delta_{i\beta}\varphi_{\alpha}-\delta_{\alpha\beta}\varphi_i,
		\end{align}
		where $\bar{R}_{l\beta\alpha i}$ denotes the curvature tensor of the round metric on $\mathbb{S}^n$. Then we have
		\begin{align*}
			A_{\alpha\beta i}-A_{i\beta\alpha}=&-\Big(\frac{\sum_{l}\varphi_l \varphi_{li}}{\varphi}-\frac{1}{2}\frac{|D\varphi|^2\varphi_i}{\varphi^2}\Big)\delta_{\alpha \beta}
			-\frac{1}{2}\Big(\varphi_i-\frac{\varphi_i}{\varphi^2}\Big)\delta_{\alpha \beta}\\
   &+\Big(\frac{\sum_{l}\varphi_l \varphi_{l\alpha}}{\varphi}-\frac{1}{2}\frac{|D\varphi|^2\varphi_\alpha}{\varphi^2}\Big)\delta_{i \beta}
			+\frac{1}{2}\Big(\varphi_\alpha-\frac{\varphi_\alpha}{\varphi^2}\Big)\delta_{i \beta}\\
   =&\frac{\varphi_{\alpha}}{\varphi}A_{\alpha\alpha}\delta_{i \beta}-\frac{\varphi_i}{\varphi}A_{ii}\delta_{\alpha\beta},
		\end{align*}
  where we used that $(\varphi_{ij})$ is diagonal at $x$. This is the equation \eqref{d2}.

  Equation \eqref{d3} was derived in \cite[Lemma 4.2]{Chen}. We include the proof for convenience of readers.
		\begin{align*}
			A_{ii\alpha\alpha}=&\varphi_{ii\alpha\alpha}
			-\frac{\sum_{k}(\varphi^{2}_{k\alpha}+\varphi_{k\alpha\alpha}\varphi_k)}{\varphi}
			+\frac{2\sum_{k}\varphi_{k\alpha}\varphi_k\varphi_{\alpha}}{\varphi^2}\\
    &+\frac{1}{2}|D\varphi|^2
			\frac{\varphi\varphi_{\alpha\alpha}-2\varphi_{\alpha}^{2}}{\varphi^3}+\frac{1}{2}\Big(\varphi_{\alpha\alpha}+\frac{\varphi_{\alpha\alpha}}{\varphi^2}-
			\frac{2\varphi_{\alpha}^{2}}{\varphi^3}\Big)\\
		=&\varphi_{ii\alpha\alpha}
			-\frac{\sum_{k}(\varphi^{2}_{k\alpha}+\varphi_{\alpha\alpha k}\varphi_k)}{\varphi}
			+\frac{\varphi_{\alpha}^{2}-|D \varphi|^2}{\varphi}
			+\frac{2\sum_{k}\varphi_{k\alpha}\varphi_k\varphi_{\alpha}}{\varphi^2}\\
   &+\frac{1}{2}|D\varphi|^2
			\frac{\varphi\varphi_{\alpha\alpha}-2\varphi_{\alpha}^{2}}{\varphi^3}
			+\frac{1}{2}\Big(\varphi_{\alpha\alpha}+\frac{\varphi_{\alpha\alpha}}{\varphi^2}-
			\frac{2\varphi_{\alpha}^{2}}{\varphi^3}\Big)\\
   =&\varphi_{ii\alpha\alpha}-\frac{\sum_{k}\varphi^{2}_{k\alpha}}{\varphi}+\frac{1}{2}|D\varphi|^2\frac{\varphi_{\alpha\alpha}}{\varphi^2} -\frac{\sum_k\varphi_{\alpha\alpha k}\varphi_k}{\varphi}\\
   &+\frac{2\sum_{k}\varphi_{k\alpha}\varphi_k\varphi_{\alpha}}{\varphi^2}-
			\frac{|D\varphi|^2\varphi_{\alpha}^{2}}{\varphi^3}+\varphi_{\alpha}^{2}\Big(\frac{1}{\varphi}-\frac{1}{\varphi^3}\Big)\\
   &-\frac{|D \varphi|^2}{\varphi}+\frac{1}{2}\Big(\varphi_{\alpha\alpha}+\frac{\varphi_{\alpha\alpha}}{\varphi^2}\Big).
		\end{align*}
		Using the definition of $A$ and the fact the $\varphi_{ij}$ is diagonal at $x$, we rewrite some terms on the right hand side of the above equation as follows:
		\begin{align*}
			&\frac{2\sum_{k}\varphi_{k\alpha}\varphi_k\varphi_{\alpha}}{\varphi^2}-
			\frac{|D\varphi|^2\varphi_{\alpha}^{2}}{\varphi^3}+\varphi_{\alpha}^{2}\Big(\frac{1}{\varphi}-\frac{1}{\varphi^3}\Big)
			=\frac{2\varphi_{\alpha}^{2}A_{\alpha\alpha}}{\varphi^2},\\
			&-\frac{\sum_{k}\varphi^{2}_{k\alpha}}{\varphi}+\frac{1}{2}|D\varphi|^2
			\frac{\varphi_{\alpha\alpha}}{\varphi^2}
			=-\frac{\varphi_{\alpha\alpha}}{\varphi}A_{\alpha\alpha}+\frac{1}{2}\varphi_{\alpha\alpha}\Big(1-\frac{1}{\varphi^2}\Big).
		\end{align*}
		Thus,
		\begin{eqnarray*}
			A_{ii\alpha\alpha}&=&\varphi_{ii\alpha\alpha}
			-\frac{\sum_{k}\varphi_{\alpha\alpha k}\varphi_k}{\varphi}
			-\frac{|D \varphi|^2}{\varphi}
			+\Big(\frac{2\varphi_{\alpha}^{2}}{\varphi^2}-\frac{\varphi_{\alpha\alpha}}{\varphi}\Big)A_{\alpha\alpha}
			+\varphi_{\alpha\alpha}.
		\end{eqnarray*}
  Interchanging the indices $i$ and $\alpha$, we have
  \begin{align}\label{s3.d2A}
      &A_{ii\alpha\alpha}-A_{\alpha\alpha ii}\nonumber\\
      =&\varphi_{ii\alpha\alpha}-\varphi_{\alpha\alpha ii}-\sum_{k}\frac{\varphi_k}{\varphi}\left(\varphi_{\alpha\alpha k}-\varphi_{iik}\right)+\Big(\frac{2\varphi_{\alpha}^{2}}{\varphi^2}-\frac{\varphi_{\alpha\alpha}}{\varphi}\Big)A_{\alpha\alpha}\nonumber\\
      &-\Big(\frac{2\varphi_{i}^{2}}{\varphi^2}-\frac{\varphi_{ii}}{\varphi}\Big)A_{ii}+\varphi_{\alpha\alpha}-\varphi_{ii}.
  \end{align}

		Taking derivatives to \eqref{s3.dalpha}, we have
  \begin{align*}
      \varphi_{\alpha\alpha ii}=&\varphi_{i\alpha\alpha i}+\delta_{i\alpha}\varphi_{\alpha i}-\delta_{\alpha\alpha}\varphi_{ii},\\
      \varphi_{ii\alpha\alpha}=&\varphi_{\alpha ii\alpha }+\delta_{\alpha i}\varphi_{i\alpha}-\delta_{ii}\varphi_{\alpha\alpha}.
  \end{align*}
  This implies that
  \begin{align}\label{s3.d2varphi}
      \varphi_{ii\alpha\alpha}-\varphi_{\alpha\alpha ii}=&\varphi_{\alpha ii\alpha }-\varphi_{i\alpha \alpha i}+\delta_{\alpha\alpha}\varphi_{ii}-\delta_{ii}\varphi_{\alpha\alpha}\nonumber\\
      =&\bar{R}_{\alpha iim}\varphi_{m\alpha}+\bar{R}_{\alpha i\alpha m}\varphi_{mi}+\delta_{\alpha\alpha}\varphi_{ii}-\delta_{ii}\varphi_{\alpha\alpha}\nonumber\\
      =&\left(\delta_{\alpha i}\delta_{im}-\delta_{\alpha m}\delta_{ii}\right)\varphi_{m\alpha}+\left(\delta_{\alpha\alpha}\delta_{im}-\delta_{\alpha m}\delta_{i\alpha}\right)\varphi_{mi}\nonumber\\
      &\quad +\delta_{\alpha\alpha}\varphi_{ii}-\delta_{ii}\varphi_{\alpha\alpha}\nonumber\\
      =&2\varphi_{ii}-2\varphi_{\alpha\alpha}.
  \end{align}
		Substituting \eqref{s3.d2varphi} into \eqref{s3.d2A} and noting the fact that
		\begin{eqnarray*}
			\varphi_{\alpha\alpha k}-\varphi_{ii k}=A_{\alpha\alpha k}-A_{ii k}, \qquad
			\varphi_{\alpha\alpha}-\varphi_{ii}=A_{\alpha\alpha}-A_{ii},
		\end{eqnarray*}
		we obtain the equation \eqref{d3}.
	\end{proof}
	
	In the case of $k=1$, equation \eqref{equa} is a quasi-linear elliptic equation on $\mathbb{S}^n$. The $C^2$ estimate follows from the $C^1$ estimate by the standard elliptic theory. In the following, we focus on the case $2\le k\le n-1$.
	\begin{proposition}\label{lc2}
		Let $f$ be a smooth, even and positive function satisfying Assumption \ref{assum}.
        There is a constant $C>0$ depending only on $n, p, k$, $\Vert f \Vert _{C^2(\mathbb{S}^n)}$ and $\min_{\mathbb{S}^n}f$, such that if $\varphi>1$ is an even,  $h$-convex solution of equation \eqref{equa}, then $\Vert \varphi \Vert _{C^2(\mathbb{S}^n)} \le C$.
	\end{proposition}
	\begin{proof}
		Since $A[\varphi]$ is positive semi-definite and we have $C^0$ and $C^1$ bounds on $\varphi$, the terms $|\varphi_{ij}|$ are controlled by
  \begin{equation*}
      H:= \text{trace}(A[\varphi])=\Delta \varphi-\frac{n}{2\varphi}|D\varphi|^2+\frac{n}{2}(\varphi-\frac{1}{\varphi}).
  \end{equation*}
  Therefore, it suffices to estimate the upper bound of $H$. Assume $H$ achieves its maximum at $x_0 \in \mathbb{S}^n$, we choose a  local orthonormal frame $e_1, e_2,...,e_n$ near $x_0$ such that $\varphi_{ij}(x_0)$ is diagonal. Denote $A=A[\varphi]$ and  $S(A)=\left(\frac{\sigma_{k}}{C_n^k}\right)^{\frac{1}{k}}(A)$. Then equation \eqref{equa} transforms to
		\begin{equation}\label{equa1}
			S(A)=\varphi^{a-1}\tilde{f},
		\end{equation}
		where $\tilde{f}=\left(\frac{f}{C^k_n}\right)^{\frac{1}{k}}$ and $a=\frac{p}{k}$.

		Taking twice derivatives of $H$ and using  \eqref{d3} of Lemma \ref{d}, we obtain
		\begin{equation}\label{hii}
			\begin{aligned}
				H_{ii}&=\sum\limits_{\alpha}A_{\alpha\alpha ii}\\
				&=\sum\limits_{\alpha} \biggl\{   A_{ii\alpha\alpha}  +\dfrac{\sum_kA_{\alpha\alpha k}\varphi_k}{\varphi}+\left[-\dfrac{2\varphi_{\alpha}^2}{\varphi^2}+\dfrac{|D\varphi|^2}{2\varphi^2}+\dfrac{1}{\varphi}A_{\alpha\alpha}+\dfrac{1}{2}\left(1+\dfrac{1}{\varphi^2}\right)\right]A_{\alpha\alpha}    \\
				&\qquad\qquad        -\dfrac{\sum_k A_{iik}\varphi_k}{\varphi}-\left[-\dfrac{2\varphi_{i}^2}{\varphi^2}+\dfrac{|D\varphi|^2}{2\varphi^2}+\dfrac{1}{\varphi}A_{ii}+\dfrac{1}{2}\left(1+\dfrac{1}{\varphi^2}\right)\right]A_{ii}
				\biggr\}\\
				&=\Delta A_{ii}+\dfrac{\sum\limits_{\alpha,k}A_{\alpha\alpha k}\varphi_k}{\varphi}-\dfrac{n\sum\limits_kA_{iik}\varphi_k}{\varphi} \\
				& \ \ \  -\sum\limits_{\alpha}\dfrac{2\varphi_{\alpha}^2}{\varphi^2}A_{\alpha\alpha}+\dfrac{1}{\varphi}\sum\limits_{\alpha}A_{\alpha\alpha}^2+\left[\dfrac{|D\varphi|^2}{2\varphi^2}+\dfrac{1}{2}\left(1+\dfrac{1}{\varphi^2}\right)\right]H \\
				& \ \ \   +\dfrac{2n\varphi_{i}^2}{\varphi^2}A_{ii}-\dfrac{n}{\varphi}A_{ii}^2-n\left[\dfrac{|D\varphi|^2}{2\varphi^2}+\dfrac{1}{2}\left(1+\dfrac{1}{\varphi^2}\right)\right]A_{ii}.
			\end{aligned}
		\end{equation}
     Since $f>0$ and $A[\varphi]\geq 0$, we know that $\varphi$ is $k$-admissible and so $(S^{ij})>0$ which is diagonal at $x_0$. We also have $(H_{ij})\le0$ at $x_0$, by the above commutator identity \eqref{hii}, it follows that at $x_0$,
		\begin{equation}\label{ineq}
			\begin{aligned}
				0\ge S^{ij}H_{ij}&=S^{ii}H_{ii}\\
				&=S^{ii}\Delta A_{ii}-\dfrac{n\sum\limits_kS_k\varphi_k}{\varphi}\\
				&\ \ +\left\{  -\sum\limits_{\alpha}\dfrac{2\varphi_{\alpha}^2}{\varphi^2}A_{\alpha\alpha}+\dfrac{1}{\varphi}\sum\limits_{\alpha}A_{\alpha\alpha}^2+\left[\dfrac{|D\varphi|^2}{2\varphi^2}+\dfrac{1}{2}\left(1+\dfrac{1}{\varphi^2}\right)\right]H  \right\}\sum\limits_iS^{ii}\\
				&\ \  +\sum\limits_i\dfrac{2n\varphi_{i}^2}{\varphi^2}S^{ii}A_{ii}-\dfrac{n}{\varphi}\sum\limits_iS^{ii}A_{ii}^2-n\left[\dfrac{|D\varphi|^2}{2\varphi^2}+\dfrac{1}{2}\left(1+\dfrac{1}{\varphi^2}\right)\right]S,
			\end{aligned}
		\end{equation}
		where we used $H_k=\sum_{\alpha}A_{\alpha\alpha k}=0$, $S_k=S^{ii}A_{iik}$ and $S=S^{ii}A_{ii}$ at $x_0$.  Moreover, we know that
  \begin{equation*}
      \sum\limits_{\alpha}A_{\alpha\alpha}^2\ge \dfrac{H^2}{n},\quad \sum\limits_{\alpha}\dfrac{2\varphi_{\alpha}^2}{\varphi^2}A_{\alpha\alpha}\le\dfrac{2|D\varphi|^2}{\varphi^2}H,\quad \text{and}\quad \sum\limits_iS^{ii}A_{ii}^2\le SH.
  \end{equation*}
  Substituting these inequalities into \eqref{ineq} implies
  \begin{align}\label{ineq2}
				0\ge& S^{ii}\Delta A_{ii}-\frac{n}{\varphi}\sum\limits_kS_k\varphi_k\nonumber\\
				&\ \ +\left\{ \frac{H^2}{n\varphi} +\left[-\dfrac{3|D\varphi|^2}{2\varphi^2}+\dfrac{1}{2}\left(1+\dfrac{1}{\varphi^2}\right)\right]H  \right\}\sum\limits_iS^{ii}\nonumber\\
				&\ \  -\dfrac{n}{\varphi}SH-n\left[\dfrac{|D\varphi|^2}{2\varphi^2}+\dfrac{1}{2}\left(1+\dfrac{1}{\varphi^2}\right)\right]S\nonumber\\
    \geq & S^{ii}\Delta A_{ii}-\frac{n}{\varphi}\sum\limits_kS_k\varphi_k+\left\{ \frac{H^2}{n\varphi} +\left(-1+\dfrac{1}{2\varphi^2}\right)H  \right\}\sum\limits_iS^{ii}\nonumber\\
				&\ \  -\dfrac{n}{\varphi}SH-n\left(1+\frac{1}{2\varphi^2}\right)S
		\end{align}
holds at $x_0$,  where we used \eqref{c1} in the last inequality.

 To estimate the first term on the right hand side of \eqref{ineq2}, we apply the Laplace operator to equation \eqref{equa1} to obtain
		\begin{align*}
			&S^{ij}A_{ijk}=(\varphi^{a-1}\tilde{f})_k,\\
           &S^{ij,rs}A_{ijk}A_{rsk}+S^{ij}\Delta A_{ij}=\Delta (\varphi^{a-1}\tilde{f}).
		\end{align*}
		The concavity of $S$ implies that
		\begin{equation}\label{ineq1}
			S^{ij}\Delta A_{ij}\ge \Delta (\varphi^{a-1}\tilde{f})
		\end{equation}
  holds at $x_0$.		Direct computation gives
		\begin{equation}\label{lap}
			\begin{aligned}
				\Delta (\varphi^{a-1}\tilde{f})=&(a-1)\varphi^{a-2}\tilde{f}\Delta\varphi+(a-1)(a-2)\varphi^{a-3}\tilde{f}|D\varphi|^2\\
    &\quad +\varphi^{a-1}\Delta\tilde{f}+2\langle D\varphi^{a-1},D\tilde{f}\rangle \\
				=&(a-1)\varphi^{a-2}\tilde{f}H+(a-1)(a-2+\frac{n}{2})\varphi^{a-3}\tilde{f}|D\varphi|^2\\
				&\quad-\dfrac{n}{2}(a-1)(\varphi-\dfrac{1}{\varphi})\varphi^{a-2}\tilde{f}+\varphi^{a-1}\Delta\tilde{f}+2\langle D\varphi^{a-1},D\tilde{f}\rangle .
			\end{aligned}			
		\end{equation}
  Substituting \eqref{ineq1}, \eqref{lap} into \eqref{ineq2}, and using $C^0,C^1$ estimates, we obtain
		\begin{equation}
			\nonumber
			0\ge\left(a_1H^2-C_1H\right)\sum\limits_iS^{ii}-C_2H-C_3,
		\end{equation}
		for some positive constants $a_1, C_0, C_1, C_2$ depending on $n, p, k$, $\Vert f \Vert _{C^2(\mathbb{S}^n)}$ and $\min_{\mathbb{S}^n}f$. 		Since
  \begin{equation*}
      \sum_i S^{ii}\ge 1
  \end{equation*}
  (see \cite[Lemma 4]{AMZ13}), if $a_1H(x_0)\ge C_1$, we have
		\begin{equation*}
			0\ge a_1H^2-(C_1+C_2)H-C_3.
		\end{equation*}
	Hence
 \begin{equation*}
     H(x_0)\le\dfrac{C_1+C_2+1}{a_1}+C_3,
 \end{equation*}
 which completes the proof.
	\end{proof}

	By the Evans-Krylov theorem and Schauder theory ( see \cite{GT}), together with Proposition \ref{lc2}, we have the following a priori estimates.
	\begin{theorem}\label{lcl}
		For each integer $l\ge0$ and $0<\alpha <1$, there exists a constant $C_l$ depending only on $n, p, k, l$, $\alpha$, $\Vert f \Vert _{C^l(\mathbb{S}^n)}$ and $\min_{\mathbb{S}^n}f$ such that
		\begin{equation}\label{ek}
			\Vert \varphi \Vert _{C^{l,\alpha}(\mathbb{S}^n)}\le C_l,
		\end{equation}
		for any even, $h$-convex solution $\varphi$ of \eqref{equa}.
	\end{theorem}

	\section{Full rank theorem}\label{sec4}
	In this section, we will prove the full rank theorem. The approach is based on the idea that the minimal eigenvalue of $A[\varphi]$  satisfies a linear differential inequality in a viscosity sense, which in turn allows to apply the strong maximum principle. An essential component of this method is a differential inequality in a viscosity sense satisfied by the smallest eigenvalue of a symmetric $2$-tensor which was established by Brendle-Choi-Daskalopoulos in \cite{BCD17}.
	\begin{lemma}\label{BK}\cite[Lemma 5]{BCD17}
 Let $T$ be a symmetric $2$-tensor and denote by $\lambda_1\le...\le\lambda_n$ the eigenvalues of $T$. Let $D\ge1$ denote the multiplicity of the smallest eigenvalue of $T$ at $x_0$, so that $\lambda_1=\cdots=\lambda_D<\lambda_{D+1}\le \cdots\le\lambda_n$. Suppose that $\psi$ is a lower support for $\lambda_1$ at $x_0$. That is, $\psi$ is a smooth function such that in an open neighborhood of $x_0$,
 \begin{equation*}
     \psi\leq \lambda_1
 \end{equation*}
and $\psi(x_0)=\lambda_1(x_0)$. Choose an orthonormal frame  such that $T$ is diagonal at $x_0$. Then, at $x_0$
		\begin{equation}\label{BK1}
			T_{ij;k}=\delta_{ij}\psi_{;k} \quad 1\le i,j\le D,
		\end{equation}
		\begin{equation}
			\psi_{;ii}\le T_{11;ii}-2\sum_{j>D}\dfrac{(T_{1j;i})^2}{\lambda_j-\lambda_1}.
		\end{equation}
	\end{lemma}
	Now we state the following full rank theorem.
	\begin{theorem}\label{FRT}
		Let $f$ be a smooth even positive function on $\mathbb{S}^n$ satisfying Assumption \ref{assum}.
		If $1<\varphi\in C^4(\mathbb{S}^n)$ is an even, $h$-convex solution to equation \eqref{equa} for $p\ge0$, then $\varphi$ is also an even,
		uniformly $h$-convex solution to the equation \eqref{equa}.
	\end{theorem}

We note that the a priori estimates established in the previous section are not used here since we are working on the assumption $\varphi\in C^4$. The full rank theorem will be used in next section for the proof of Theorem \ref{Main}.


	\begin{proof}[Proof of Theorem \ref{FRT}]	

		For each fixed $x_0$, we choose a local orthonormal frame $e_1,
		..., e_n$ so that $(\varphi_{ij})$ is diagonal at $x_0$, with $A_{ii}=\lambda_i$ for
		$i=1, 2, ..., n$ arranged in increasing order. Differentiating the equation \eqref{equa} and using the summation convention, we have
		\begin{equation}\nonumber
			(\varphi^{p-k}f)_{ab}=(\sigma_k)_{ab}=\sigma_k^{ij,rs}A_{ija}A_{rsb}+\sigma_{k}^{ij}A_{ijab}.
		\end{equation}
		By using \eqref{d3} and $\sigma_k^{ij}$ is diagonal at $x_0$, we obtain at $x_0$
		\begin{equation}\label{Aii11}
			\begin{aligned}
				(\varphi^{p-k}f)_{11}&=\sigma_k^{ij,rs}A_{ij1}A_{rs1}+\sigma_k^{ii}A_{ii11}\\
				&=\sigma_k^{ij,rs}A_{ij1}A_{rs1}+\sigma_k^{ii}A_{11ii}-\sum_{i,m}\sigma_k^{ii}\dfrac{A_{11m}\varphi_m}{\varphi}\\
				&\ -\sum_{i}\sigma_k^{ii}\left[-\dfrac{2\varphi_{1}^2}{\varphi^2}+\dfrac{|D\varphi|^2}{2\varphi^2}+\dfrac{1}{\varphi}A_{11}+\dfrac{1}{2}\left(1+\dfrac{1}{\varphi^2}\right)\right]A_{11} +\sum_m\dfrac{(\varphi^{p-k}f)_m\varphi_m}{\varphi}\\
				&\ -2\sum_i\frac{\varphi_i^2}{\varphi^2}\sigma_k^{ii}A_{ii}+k\left[\dfrac{|D\varphi|^2}{2\varphi^2}+\dfrac{1}{2}\left(1+\dfrac{1}{\varphi^2}\right)\right]\varphi^{p-k}f+\sum_i\frac{1}{\varphi}\sigma_k^{ii}A^2_{ii},
			\end{aligned}
		\end{equation}
		where we also used
  \begin{align*}
      (\varphi^{p-k}f)_m=&\left(\sigma_k\right)_m=\sum_i\sigma_k^{ii} A_{iim},\\
      \sigma_k^{ii}A_{ii}=&k\sigma_k=k\varphi^{p-k}f
  \end{align*}
  in the second equality.
		
		Now we deduce an inequality for the lowest eigenvalue $\lambda_1$ of $A=A[\varphi]$ in a viscosity sense. Let $\psi$ be a smooth lower support at $x_0$  for $\lambda_1$, and let $D\ge1$ denote the multiplicity of $\lambda_1(x_0)$, so that $\psi(x_0)=\lambda_1(x_0)=A_{11}(x_0)=\cdots=A_{DD}(x_0)$. We denote the complement of the set $\{i,j,k,l>D\}$ in $\{1,...,n\}^4$ by $\Lambda$.
		
		For convenience, we adopt the notation of Caffarelli-Friedman \cite{CF85} and Guan-Ma \cite{GM03}.
		For any two functions defined on $\mathbb{S}^n$ and $x_0 \in
		\mathbb{S}^n$, we say that $h\lesssim m$ provided there exist positive
		constants $c_1$ and $c_2$ such that
		\begin{eqnarray*}
			h(x_0)-m(x_0)\leq c_1|D \psi(x_0)|+c_2\psi(x_0)
		\end{eqnarray*}
		with $c_1$ and $c_2$ depending only on $|\varphi|_{C^3}$,
		$|f|_{C^{2}}$ and $n$ (independent of $x_0$).
		We write $h\sim m$ if $h\lesssim m$ and $h\lesssim m$.
		
		By Lemma \ref{BK}, we have
		\begin{equation}\label{sim0}
			\begin{cases}
				\displaystyle A_{ii}\sim 0\, & \text{if}\quad 1\le i\le D;\\[2.5ex]
				\displaystyle A_{ijl}\sim 0\, & \text{if}\quad 1\le i,j\le D,\ 1\le l\le n;\\[2.5ex]
				\displaystyle \psi_{ii}\le A_{11ii}-2\sum_{j>D}\dfrac{(A_{1ji})^2}{\lambda_j}.
			\end{cases}
		\end{equation}
		Using \eqref{Aii11} and \eqref{sim0}, we get at $x_0$
		\begin{equation}\label{G0}
			\begin{aligned}
				\sigma_k^{ij}\psi_{ij}(x_0)&=\sigma_k^{ii}\psi_{ii}\\
				&\le \sigma_k^{ii}A_{11ii}-2\sum_{j>D}\frac{\sigma_k^{ii}}{\lambda_j}(A_{1ji})^2\\
				&\lesssim (\varphi^{p-k}f)_{11}-2\sum_{j>D}\frac{\sigma_k^{ii}}{\lambda_j}(A_{1ji})^2-\sigma_k^{ij,rs}A_{ij1}A_{rs1}+2\sum_{i>D
				}\frac{\varphi_i^2}{\varphi^2}\sigma_k^{ii}A_{ii}\\
				&\ \ -\sum_{m}\dfrac{(\varphi^{p-k}f)_m\varphi_m}{\varphi}-k\left[\dfrac{|D\varphi|^2}{2\varphi^2}+\dfrac{1}{2}\left(1+\dfrac{1}{\varphi^2}\right)\right]\varphi^{p-k}f.
			\end{aligned}
		\end{equation}
		
		We first estimate the second to fourth terms on the right hand side of \eqref{G0}. As mentioned earlier, we must proceed cautiously due to the non-Codazzi nature of $A[\varphi]$. We begin with commuting the covariant derivatives and making use of basic properties of the elementary symmetric functions.

  \textbf{Claim 1.} We have
\begin{align}\label{s4.cla1}
    -2\sum_{j>D}\frac{\sigma_k^{ii}}{\lambda_j}(A_{1ji})^2-\sigma_k^{ij,rs}A_{ij1}A_{rs1}+2\sum_{i>D}\frac{\varphi_i^2}{\varphi^2}\sigma_k^{ii}A_{ii}
   \lesssim  &-\dfrac{k+1}{k}\dfrac{(\varphi^{p-k}f)_1^2}{\varphi^{p-k}f}.
\end{align}
  \begin{proof}[Proof of Claim 1]
  First, we regroup the summations in these terms as
		\begin{equation*}
			2\sum_{j>D}\frac{\sigma_k^{ii}}{\lambda_j}(A_{1ji})^2=2\sum_{i,j>D}\frac{\sigma_k^{ii}}{\lambda_j}(A_{1ji})^2+2\sum_{j>D,i\le D}\frac{\sigma_k^{ii}}{\lambda_j}(A_{1ji})^2,
		\end{equation*}
		and
		\begin{equation*}
			\sigma_k^{ij,rs}A_{ij1}A_{rs1}=\sum_{i,j,r,s>D}\sigma_k^{ij,rs}A_{ij1}A_{rs1}+\sum_{(i,j,r,s)\in \Lambda}\sigma_k^{ij,rs}A_{ij1}A_{rs1}.
		\end{equation*}
	For any $i,j>D$, by \eqref{d2} we have
  \begin{equation}\label{s4.A1}
      A_{1ji}=A_{ij1}+\dfrac{\varphi_1}{\varphi}A_{11}\delta_{ij}\sim A_{ij1}.
  \end{equation}
  Since $\sigma_k^{ij}$ is diagonal at $x_0$, using the Lemma \ref{sig} and \eqref{sim0}, \eqref{s4.A1}, we obtain that
		\begin{equation}\label{G1}
			\begin{aligned}
				&-2\sum_{i,j>D}\frac{\sigma_k^{ii}}{\lambda_j}(A_{1ji})^2-\sum_{i,j,r,s>D}\sigma_k^{ij,rs}A_{ij1}A_{rs1}\\
    \sim &-2\sum_{i,j>D}\frac{\sigma_k^{ii}}{\lambda_j}(A_{ij1})^2-\sum_{i,j,r,s>D}\sigma_k^{ij,rs}A_{ij1}A_{rs1}\\
				\lesssim&-\dfrac{k+1}{k}\dfrac{\Big(\sum\limits_{i,j>D}\sigma_k^{ij}A_{ij1}\Big)^2}{\sigma_k}=-\dfrac{k+1}{k}\dfrac{\Big(\sum\limits_{i>D}\sigma_k^{ii}A_{ii1}\Big)^2}{\sigma_k}\\
				\sim&-\dfrac{k+1}{k}\dfrac{(\varphi^{p-k}f)_1^2}{\varphi^{p-k}f}.
			\end{aligned}
		\end{equation}

		On the other hand, by Lemma \ref{AA}, we have
		\begin{equation*}
			\begin{aligned}
				\sum_{(i,j,r,s)\in \Lambda}\sigma_k^{ij,rs}A_{ij1}A_{rs1}&=2\sum_{j>D,i\le D}\sigma_{k-2}(A|ij)A_{ii1}A_{jj1}-2\sum_{j>D,i\le D}\sigma_{k-2}(A|ij)A_{ij1}^2\\
				&\sim-2\sum_{j>D,i\le D}\sigma_{k-2}(A|ij)A_{ij1}^2,
			\end{aligned}
		\end{equation*}
  where we used $A_{ii1}\sim 0$ for $i\leq D$. Note that for any $i\le D,j>D$,
  \begin{align*}
      A_{1ji}=&A_{ij1}+\frac{\varphi_1}{\varphi}A_{11}\delta_{ij}-\frac{\varphi_i}{\varphi}A_{ii}\delta_{1j}=A_{ij1}\\
      A_{ij1}=&A_{ji1}=A_{1ij}+\dfrac{\varphi_j}{\varphi}A_{jj}\delta_{i1}\sim\dfrac{\varphi_j}{\varphi}A_{jj}\delta_{i1}
  \end{align*}
  by \eqref{d2} and \eqref{sim0}. Moreover, by Lemma \ref{s2.lem1} we also have
  \begin{align*}
      &\sigma_{k-1}(A|1)-\sigma_{k-2}(A|j1)A_{jj}=\sigma_{k-1}(A|1j),\\
      &\sigma_{k-1}(A|1j)=\sigma_{k-1}(A|j)-\sigma_{k-2}(A|1j)A_{11}\sim\sigma_{k-1}(A|j)=\sigma_k^{jj}.
  \end{align*}
  So that we have
		\begin{equation}\label{G2}
			\begin{aligned}
				&-2\sum_{j>D,i\le D}\frac{\sigma_k^{ii}}{\lambda_j}(A_{1ji})^2-\sum_{(i,j,r,s)\in \Lambda}\sigma_k^{ij,rs}A_{ij1}A_{rs1}\\
				\sim&-2\sum_{j>D}\Big(\sigma_{k-1}(A|1)-\sigma_{k-2}(A|j1)A_{jj}\Big)\dfrac{\varphi_j^2}{\varphi^2}A_{jj}\\
				\sim&-2\sum_{j>D}\sigma_{k}^{jj}\dfrac{\varphi_j^2}{\varphi^2}A_{jj}.
			\end{aligned}
		\end{equation}
Combining \eqref{G1} and \eqref{G2} proves the \textbf{Claim 1}.
\end{proof}

	Plugging \eqref{s4.cla1} into \eqref{G0}, we obtain the following inequality
		\begin{equation}\label{s4.pf1}
			\begin{aligned}
				\sigma_k^{ij}\psi_{ij}(x_0)\lesssim& (\varphi^{p-k}f)_{11} -\dfrac{k+1}{k}\dfrac{(\varphi^{p-k}f)_1^2}{\varphi^{p-k}f}-\sum_{m}\dfrac{(\varphi^{p-k}f)_m\varphi_m}{\varphi}\\
                &-k\left[\dfrac{|D\varphi|^2}{2\varphi^2}+\dfrac{1}{2}\left(1+\dfrac{1}{\varphi^2}\right)\right]\varphi^{p-k}f.
			\end{aligned}
		\end{equation}
  We estimate the terms separately:
		\begin{align}
			(\varphi^{p-k}f)_{1}^2=&\Big((p-k)\varphi^{p-k-1}\varphi_1 f+\varphi^{p-k}f_1\Big)^2\nonumber\\
			=& (\varphi^{p-k}f)^2\left((p-k)^2\frac{\varphi_{1}^{2}}{\varphi^{2}}+\frac{f_{1}^{2}}{ f^2}+2(p-k)\frac{\varphi_1f_{1}}{\varphi f}\right),\label{s4.t1}\\
   -\dfrac{\Sigma_m(\varphi^{p-k}f)_m\varphi_m}{\varphi}=&\varphi^{p-k}f\left((k-p)\dfrac{|D\varphi|^2}{\varphi^2}-\dfrac{\sum_mf_m\varphi_m}{\varphi f}\right), \label{s4.t2}
		\end{align}
		and
			\begin{align} \label{s4.t3}
				(\varphi^{p-k}f)_{11}&=\varphi^{p-k}f\left((p-k)\dfrac{\varphi_{11}}{\varphi}+\dfrac{f_{11}}{f}+2(p-k)\dfrac{\varphi_1f_1}{\varphi f}+(p-k)(p-k-1)\dfrac{\varphi_1^2}{\varphi^2}\right)\nonumber\\
				&\sim\varphi^{p-k}f\bigg(\dfrac{p-k}{2}\dfrac{|D\varphi|^2}{\varphi^2}-\frac{p-k}{2}(1-\dfrac{1}{\varphi^2})+\dfrac{f_{11}}{f}\nonumber\\
                &\qquad\qquad\quad +2(p-k)\dfrac{\varphi_1f_1}{\varphi f}+(p-k)(p-k-1)\dfrac{\varphi_1^2}{\varphi^2}\biggr),
			\end{align}
where we used the definition \eqref{s1.Aij} of $A_{11}$. Substituting \eqref{s4.t1} -- \eqref{s4.t3} into \eqref{s4.pf1}, we arrive at
		\begin{equation}\label{fin}
			\begin{aligned}
				\sigma_k^{ij}\psi_{ij}(x_0)&\lesssim\varphi^{p-k}f\left(\dfrac{f_{11}}{f}-\dfrac{k+1}{k}\dfrac{f_1^2}{f^2}-\dfrac{\sum_{m}\varphi_mf_m}{\varphi f}-\frac{2(p-k)}{k}\dfrac{\varphi_1f_1}{\varphi f}\right.\\
				&\quad\qquad\qquad\left.-\frac{p(p-k)}{k}\frac{\varphi_1
					^2}{\varphi^2}-\frac{p}{2}(1+\dfrac{|D\varphi|^2}{\varphi^2                      })+\frac{p-2k}{2}\dfrac{1}{\varphi^2}\right).
			\end{aligned}
		\end{equation}
		\vspace{0.5em}

	\textbf{Claim 2.} If $f$ satisfies the Assumption \ref{assum}, then $\sigma_k^{ij}\psi_{ij}(x_0)\lesssim 0$.
	\begin{proof}[Proof of Claim 2]
		We divide the proof into five cases.
		
		\textbf{Case (1).} $p=0$. In this case, we have that
		\begin{equation}
			\nonumber
			\begin{aligned}
				\sigma_k^{ij}\psi_{ij}(x_0)&\lesssim\varphi^{p-k}f\left(\dfrac{f_{11}}{f}-\dfrac{k+1}{k}\dfrac{f_1^2}{f^2}-\dfrac{\varphi_1f_1-\sum_{m\neq 1}\varphi_mf_m}{\varphi f}-\dfrac{k}{\varphi^2}\right)\\
				&\lesssim\varphi^{-k}f\left(\dfrac{f_{11}}{f}-\dfrac{k+1}{k}\dfrac{f_1^2}{f^2}+\dfrac{|D\varphi||Df|}{\varphi f}-\dfrac{k}{\varphi^2}\right)\\
				&=k\varphi^{-k}f^{1+\frac{1}{k}}\left(-(f^{-\frac{1}{k}})_{11}+\dfrac{|D\varphi|}{\varphi}|Df^{-\frac{1}{k}}|-\dfrac{f^{-\frac{1}{k}}}{\varphi^2}\right).
			\end{aligned}
		\end{equation}
		By \eqref{c0} and \eqref{c1}, we have
		\begin{equation}\nonumber
			\varphi^2\le\left(\max\limits_{\mathbb{S}^n}\varphi\right)^2\le2+8\max\limits_{\mathbb{S}^n}\left(\dfrac{f}{C^k_n}\right)^{\frac{1}{k}},\quad \dfrac{|D\varphi|}{\varphi}\le 1.
		\end{equation}
		Thus, we obtain
		\begin{equation}
			\nonumber
			\begin{aligned}
				\sigma_k^{ij}\psi_{ij}(x_0)&\lesssim -k\varphi^{-k}f^{1+\frac{1}{k}}\left((f^{-\frac{1}{k}})_{11}-|Df^{-\frac{1}{k}}|+\dfrac{f^{-\frac{1}{k}}}{2+8\max\limits_{\mathbb{S}^n}\left(\dfrac{f}{C^k_n}\right)^{\frac{1}{k}}}\right)\le 0.
			\end{aligned}
		\end{equation}
		This completes the proof of Case (1).
		
		\textbf{Case (2).} $0< p\le \frac{k}{2}$. In this case $\dfrac{p(p-k)}{k}+\dfrac{p}{2}=\dfrac{p(2p-k)}{2k}\leq 0$. Using Cauchy-Schwarz inequality and \eqref{c1} we obtain
		\begin{equation}\label{case23}
			\begin{aligned}
				&-\dfrac{\sum_{m}\varphi_mf_m}{\varphi f}-\frac{2(p-k)}{k}\dfrac{\varphi_1f_1}{\varphi f}-\frac{p(p-k)}{k}\frac{\varphi_1^2}{\varphi^2}-\frac{p}{2}(1+\dfrac{|D\varphi|^2}{\varphi^2                      })+\frac{p-2k}{2}\dfrac{1}{\varphi^2}\\
				\le&\frac{k+2p}{k}\dfrac{|D\varphi||Df|}{\varphi f}-\Big(\frac{p(p-k)}{k}+\frac{p}{2}\Big)\frac{|D\varphi|^2}{\varphi^2}-\frac{p}{2}\\
				\le&\frac{k+2p}{k}\dfrac{|Df|}{f}-\dfrac{p^2}{k}.
			\end{aligned}
		\end{equation}
		Therefore
		\begin{equation}\nonumber
			\begin{aligned}
				\sigma_k^{ij}\psi_{ij}(x_0)&\lesssim\varphi^{p-k}f\left(\dfrac{f_{11}}{f}-\dfrac{k+1}{k}\dfrac{f_1^2}{f^2}+\frac{k+2p}{k}\dfrac{|Df|}{f}-\dfrac{p^2}{k}\right)\\
				&=-k\varphi^{p-k}f^{1+\frac{1}{k}}\left((f^{-\frac{1}{k}})_{11}-\frac{k+2p}{k}|Df^{-\frac{1}{k}}|+\dfrac{p^2}{k^2}f^{-\frac{1}{k}}\right)\\
				&\le 0.
			\end{aligned}
		\end{equation}
		Then we finish the proof of Case (2).

		\textbf{Case (3).} $\frac{k}{2}< p<k$. In this case $\dfrac{p(p-k)}{k}+\dfrac{p}{2}=\dfrac{p(2p-k)}{2k}\ge0$. Recall \eqref{case23}, we have
		\begin{equation}\nonumber
			\begin{aligned}
				&-\dfrac{\sum_{m}\varphi_mf_m}{\varphi f}-\frac{2(p-k)}{k}\dfrac{\varphi_1f_1}{\varphi f}-\frac{p(p-k)}{k}\frac{\varphi_1^2}{\varphi^2}-\frac{p}{2}(1+\dfrac{|D\varphi|^2}{\varphi^2                      })-\frac{2k-p}{2}\dfrac{1}{\varphi^2}\\
				\le&\frac{3k-2p}{k}\dfrac{|D\varphi||Df|}{\varphi f}-\Big(\frac{p(p-k)}{k}+\frac{p}{2}\Big)\frac{|D\varphi|^2}{\varphi^2}-\frac{p}{2}\\
				\le&\dfrac{(3k-2p)^2}{2kp(2p-k)}\dfrac{|Df|^2}{f^2}-\dfrac{p}{2},
			\end{aligned}
		\end{equation}
  where we used Cauchy-Schwarz inequality in the last inequality. Therefore
		\begin{equation}\nonumber
			\begin{aligned}
				\sigma_k^{ij}\psi_{ij}(x_0)&\lesssim\varphi^{p-k}f\left(\dfrac{f_{11}}{f}-\dfrac{k+1}{k}\dfrac{f_1^2}{f^2}+\dfrac{(3k-2p)^2}{2kp(2p-k)}\dfrac{|Df|^2}{f^2}-\dfrac{p}{2}\right)\\
				&=-k\varphi^{p-k}f^{1+\frac{1}{k}}\left((f^{-\frac{1}{k}})_{11}-\dfrac{(3k-2p)^2}{p(2p-k)}\dfrac{|Df^{-\frac{1}{k}}|^2}{2f^{-\frac{1}{k}}}+\dfrac{1}{2}\dfrac{p}{k}f^{-\frac{1}{k}}\right)\\
				&\le 0.
			\end{aligned}
		\end{equation}
		Thus we finish the proof of Case (3).
		
		\textbf{Case (4).} $k\le p\le 2k$. In this case $\dfrac{p+1}{p}-\dfrac{k+1}{k}=-\dfrac{p-k}{pk}\le 0$, we realize that
		\begin{equation}\label{mix}
			-\dfrac{p-k}{pk}\dfrac{f_1^2}{f^2}-\frac{2(p-k)}{k}\dfrac{\varphi_1f_1}{\varphi f}-\frac{p(p-k)}{k}\frac{\varphi_1^2}{\varphi^2}=-\dfrac{p-k}{pk}\Big(\dfrac{f_1}{f}+p\frac{\varphi_1}{\varphi}\Big)^2\le0.
		\end{equation}
		Substituting \eqref{mix} into \eqref{fin} and using Cauchy-Schwarz inequality, we obtain
		\begin{equation}\nonumber
			\begin{aligned}
				\sigma_k^{ij}\psi_{ij}(x_0)&\lesssim\varphi^{p-k}f\left(\dfrac{f_{11}}{f}-\dfrac{p+1}{p}\dfrac{f_1^2}{f^2}+\dfrac{|D\varphi||Df|}{\varphi f}-\frac{p}{2}(1+\dfrac{|D\varphi|^2}{\varphi^2      })\right)\\
				&\le\varphi^{p-k}f\left(\dfrac{f_{11}}{f}-\dfrac{p+1}{p}\dfrac{f_1^2}{f^2}+\dfrac{|Df|^2}{2pf^2}-\frac{p}{2}\right)\\
				&=-p\varphi^{p-k}f^{1+\frac{1}{p}}\left((f^{-\frac{1}{p}})_{11}-\dfrac{|Df^{-\frac{1}{p}}|^2}{2f^{-\frac{1}{p}}}+\frac{1}{2}f^{-\frac{1}{p}}\right)\\
				&\le 0.
			\end{aligned}
		\end{equation}
		Thus we complete the proof of Case (4).

        \textbf{Case (5).} $p>2k$. This case is similar to Case (4), the only difference is $\dfrac{p-2k}{2}\dfrac{1}{\varphi^2}>0$ in this case. Using $\varphi>1$, we have
        \begin{equation}\nonumber
			\begin{aligned}
				\sigma_k^{ij}\psi_{ij}(x_0)&\lesssim\varphi^{p-k}f\left(\dfrac{f_{11}}{f}-\dfrac{p+1}{p}\dfrac{f_1^2}{f^2}+\dfrac{|Df|^2}{2pf^2}-\frac{p}{2}+\dfrac{p-2k}{2}\dfrac{1}{\varphi^2}\right)\\
                &\le\varphi^{p-k}f\left(\dfrac{f_{11}}{f}-\dfrac{p+1}{p}\dfrac{f_1^2}{f^2}+\dfrac{|Df|^2}{2pf^2}-k\right)\\
				&=-p\varphi^{p-k}f^{1+\frac{1}{p}}\left((f^{-\frac{1}{p}})_{11}-\dfrac{|Df^{-\frac{1}{p}}|^2}{2f^{-\frac{1}{p}}}+\frac{k}{p}f^{-\frac{1}{p}}\right)\\
				&\le 0.
			\end{aligned}
		\end{equation}
        Thus we complete the proof of Case (5), and so the \textbf{Claim 2}  is verified.
		\end{proof}
		
		Therefore, we in fact deduce the following inequality:
		\begin{equation}\nonumber
			\sigma_k^{ij}\psi_{ij}(x_0)\le C(\psi(x_0)+|D\psi(x_0)|),
		\end{equation}
		where $C$ is a  positive constant independent of $x_0$. Since $f$ is positive and $A[\varphi]$ is positive semi-definite, we know that $\varphi$ is $k$-admissible and then $\sigma_k^{ij}$ is positive definite. Then the strong maximum principle for viscosity solutions (cf. \cite{BD99}) implies that the set $\{\lambda_1=0\}$ is open. Hence, if $\lambda_1$ was zero somewhere, it would be zero everywhere. However, we know $A[\varphi]$ is positive definite at the minimum point of $\varphi$. Therefore, we have $A[\varphi]>0$ everywhere which completes the proof of Theorem \ref{FRT}.
	\end{proof}
	
	\section{Proof of Theorem \ref{Main}}\label{sec5}
	
	In this section, we apply the degree theory developed in \cite{Li89} to prove Theorem \ref{Main}. We first need the following uniqueness of constant solutions to equation \eqref{equa} with constant prescribed function. 	
	\begin{lemma}\label{A-C}
		Let $n\ge2$ and $1\le k\le n-1$ be integers, let $p\ge0$ be a real number, and let $\gamma$ be a positive constant. Consider the equation
		\begin{eqnarray}\label{MA-c}
			\sigma_k(A[\varphi])=\varphi^{p-k}\gamma.
		\end{eqnarray}
        Then we have following conclusions

        (1) If $0\le p<2k$, then the equation \eqref{MA-c} has a unique even uniformly $h$-convex solution $\varphi=c>1$ for each $\gamma\in (0,\infty)$.

        (2) If $p=2k$ and $0<\gamma<2^{-k}{C_n^k}$, then the equation \eqref{MA-c} has a unique even uniformly $h$-convex solution $\varphi=c>1$.

        (3) If $p>2k$ and $\gamma=\gamma_p=p^{-\frac{p}{2}}k^k(p-2k)^{\frac{p-2k}{2}}C^k_n$, then the equation \eqref{MA-c} has a unique even uniformly $h$-convex solution $\varphi(x)=\Big(\frac{p}{p-2k}\Big)^{\frac{1}{2}}$.

	\end{lemma}
	
	\begin{proof}
		This lemma is a corollary of Theorem 8.1 in \cite{Li-Xu}.
	\end{proof}
	
	Let consider the Banach space
	\begin{eqnarray*}
		\mathcal{B}^{2,\alpha}(\mathbb{S}^n)=\{\varphi \in
		C^{2,\alpha}(\mathbb{S}^n): \varphi \ \mbox{is even}\}
	\end{eqnarray*}
	and
	\begin{eqnarray*}
		\mathcal{B}_{0}^{4,\alpha}(\mathbb{S}^n)=\{\varphi \in
		C^{4,\alpha}(\mathbb{S}^n): A[\varphi]\ge0 \ \mbox{and} \ 	\varphi \ \mbox{is even}\}.
	\end{eqnarray*}
	Given a smooth positive even function $f$ on $\mathbb{S}^n$, consider the nonlinear differential operator $\mathcal{L}(\cdot, t): \mathcal{B}_{0}^{4,\alpha}(\mathbb{S}^n)\rightarrow 	\mathcal{B}^{2,\alpha}(\mathbb{S}^n)$ defined by
	\begin{eqnarray*}
		\mathcal{L}(\varphi, t)=\sigma_k(A[\varphi])-\varphi^{p-k} f_t,
	\end{eqnarray*}
	where
	\begin{equation}\label{ft}
		f_t=\begin{cases}
			\displaystyle \Big((1-t)(\max_{\mathbb{S}^n}f)^{-\frac{1}{k}}+t f^{-\frac{1}{k}}\Big)^{-k}, & \text{if}\quad 0\le p<k;\\[2.5ex]
			\displaystyle \Big((1-t)(\max_{\mathbb{S}^n}f)^{-\frac{1}{p}}+t f^{-\frac{1}{p}}\Big)^{-p}, & \text{if} \quad k\le p\le2k;\\[2.5ex]
            \displaystyle \Big((1-t)\gamma_p^{-\frac{1}{p}}+t f^{-\frac{1}{p}}\Big)^{-p}, & \text{if} \quad p>2k,\quad 2\leq k\leq n-1.
		\end{cases}
	\end{equation}
	It is straightforward to verify that if $f$ satisfies Assumption \ref{assum}, then $f_t$ also satisfies Assumption \ref{assum} for all $t\in[0,1]$.
	
	Consider the open set
 $$\mathcal{O}_R=\{\varphi \in \mathcal{B}_{0}^{4,\alpha}(\mathbb{S}^n):
	1+\frac{1}{R}< \varphi, \ 0< A[\varphi], \ |\varphi|_{C^{4,\alpha}(\mathbb{S}^n)}<R\}$$
 in $\mathcal{B}_{0}^{4,\alpha}(\mathbb{S}^n)$.  The a priori estimates in Section \ref{sec3} and full rank theorem in Section \ref{sec4} imply that $\mathcal{L}(\varphi, t)=0$ has no solution on $\partial \mathcal{O}_R$ for sufficiently large $R$. In fact, suppose there is an even, $h$-convex solution $\varphi\in\partial \mathcal{O}_R$ for $\mathcal{L}(\varphi,t)=0$. The full rank theorem implies that $\varphi$ is uniformly $h$-convex. Applying the $C^0$ estimate in Lemma \ref{lc0} and the regularity theorem in Theorem \ref{lcl}, we have that the lower bound of $\varphi$ is controlled by a constant $C>1$ depending only the lower bound of $f$, $n,p,k$;  and the $C^{4,\alpha}$ norm of $\varphi$ is controlled by $C^{4}$ norm of $f$, lower bound of $f$ and $n,p,k, \alpha$.  By choosing $R$ sufficiently large, we have that such a uniformly $h$-convex solution $\varphi$ satisfies $\varphi>1+\frac{1}{R}$ and the $C^{4,\alpha}$ norm of $\varphi$ can be strictly less than $R$. This is a contradiction to that $\varphi\in \partial \mathcal{O}_R$  and so $\partial \mathcal{O}_R\cap \mathcal{L}^{-1}(\cdot,t)=\emptyset$. Therefore the degree $\deg(\mathcal{L}(\cdot, t), \mathcal{O}_R, 0)$ is well-defined for each $0\leq t\leq 1$ (\cite[Definition 2.2]{Li89}) and the homotopic invariance of the degree ( \cite[Proposition 2.2]{Li89}) implies
	\begin{eqnarray}\label{hot}
		\deg(\mathcal{L}(\cdot, 1), \mathcal{O}_R, 	0)=\deg(\mathcal{L}(\cdot, 0), \mathcal{O}_R, 0).
	\end{eqnarray}
	
	At $t=0$, $\mathcal{L}(\varphi, 0)=0$ is equivalent to $\sigma_k(A[\varphi])=\varphi^{p-k} f_0$, where $f_0$ is a positive constant given by
 \begin{equation}\label{f0}
		f_0=\begin{cases}
			\max_{\mathbb{S}^n}f, & \text{if}\quad 0\le p\leq 2k;\\[2ex]
			\gamma_p, & \text{if} \quad p>2k.
		\end{cases}
	\end{equation}
 By assumption, for $p=2k$, we have $\max_{\mathbb{S}^n}f<2^{-k}C_n^k$. Then Lemma \ref{A-C} tells us that
	$\varphi=c>1$ is the unique even solution for $\mathcal{L}(\varphi, 0)=0$ in $\mathcal{O}_R$. In particular, $c=\Big(\frac{p}{p-2k}\Big)^{\frac{1}{2}}$ when $p>2k$. We can compute the degree $\deg(\mathcal{L}(\cdot, 0), \mathcal{O}_R, 0)$ in terms of the degree of the linearized operator $L_{c}$ of $\mathcal{L}(\varphi,0)$ at $\varphi=c$. Direct calculation shows that
	at $\varphi=c$ the linearized operator is given by
	\begin{align*}
		L_{c}(\bar{\varphi})=&\frac{d}{d\varepsilon}\bigg|_{\varepsilon=0}\mathcal{L}(c+\varepsilon\bar{\varphi},0)\\
  =&\sigma_k^{ij}\left(A[c]\right)\frac{d}{d\varepsilon}\bigg|_{\varepsilon=0}A_{ij}[c+\varepsilon\bar{\varphi}]-(p-k)c^{p-k-1}f_0\bar{\varphi}\\
  =&C_n^{k-1}\frac{1}{2^{k-1}}(c-\frac{1}{c})^{k-1}\left(\Delta_{\mathbb{S}^n}\bar{\varphi}+\frac{n}{2}\left(1+\frac{1}{c^2}\right)\bar{\varphi}\right)-\frac{(p-k)}{c}\sigma_k(A[c])\bar{\varphi}\\
  =&a\left(\Delta_{\mathbb{S}^n}+b\right)\bar{\varphi},
	\end{align*}
	where $a=C_n^{k-1}\frac{1}{2^{k-1}}(c-\frac{1}{c})^{k-1}>0$ and $0<
	b=n-\Big(k-1+\frac{p(n-k+1)}{k}\Big)(\frac{1}{2}-\frac{1}{2c^2})\le n$ (If $k=1$ and $p>2k$, then $b=0$, we do not consider this case). The linearized operator $L_{c}$ is invertible, since $\bar{\varphi}=0$ is the unique even solution of $L_{c}(\bar{\varphi})=0$. 	 So, by Proposition
	2.3 in \cite{Li89} we have
	\begin{eqnarray*}
		\deg(\mathcal{L}(\cdot, 0), \mathcal{O}_R, 0)=\deg(L_{c}, 	\mathcal{O}_R, 0).
	\end{eqnarray*}
	As the eigenvalues of the Laplace operator $\Delta_{\mathbb{S}^n}$ on $\mathbb{S}^n$ are strictly less than
	$-n$ except for the first two eigenvalues $0$ and $-n$, there is only one positive eigenvalue of $L_{c}$
	with multiplicity $1$, namely $\mu=ab$. 	Then we have by Proposition
	2.4 in \cite{Li89}
	\begin{eqnarray*}
		\deg(\mathcal{L}(\cdot, 0), \mathcal{O}_R, 0)=\deg(L_{c}, \mathcal{O}_R, 0)=\sum_{\mu_j>0}(-1)^{\beta_j}=-
		1,
	\end{eqnarray*}
 where $\mu_j$ are eigenvalues of $L_c$ and $\beta_j$ its multiplicity.  Therefore, it follows from \eqref{hot}
	\begin{eqnarray*}
		\deg(\mathcal{L}(\cdot, 1), \mathcal{O}_R; 	0)=\deg(\mathcal{L}(\cdot, 0), \mathcal{O}_R, 0)=-
		1.
	\end{eqnarray*}
	In particular, for $t=1$ we obtain a $C^4$, even and uniformly $h$-convex solution to $\mathcal{L}(\varphi,1)=0$, setting the existence result of Theorem \ref{Main}. The regularity of $\varphi$ follows from Theorem \ref{lcl}. This completes the proof of Theorem \ref{Main}.

	
\end{document}